\documentclass[a4paper, 10pt, intlimits]{amsart}
\usepackage[utf8x]{inputenc}

\usepackage[unicode]{hyperref}
\usepackage{graphicx}
\usepackage{tikz}
\allowdisplaybreaks


\newtheorem{theorem}{Theorem}[section]
\newtheorem{proposition}[theorem]{Proposition}
\newtheorem{lemma}[theorem]{Lemma}

\newtheorem{remark}[theorem] {Remark}

\makeatletter

\hypersetup{
breaklinks=true,
colorlinks=true,
linkcolor=blue,
citecolor=blue,
urlcolor=blue,
}

\numberwithin{equation}{section}

\newcommand{\dist}{\beta}

\title[The Gradient Flow of the M\"obius Energy]{The Gradient Flow of the M\"obius energy:\\ $\varepsilon$-regularity and consequences}

\author{
Simon Blatt}
\address[Simon Blatt]{Paris Lodron Universit\"at Salzburg, Hellbrunner Strasse 34, 5020 Salzburg, Austria}
\email{simon.blatt@sbg.ac.at}
\subjclass[2010]{53C44, 35S10}

\begin{document}
\date{\today}


\begin{abstract}
In this article we study the gradient flow of the 
M\"obius energy introduced by O'Hara in 1991 \cite{OHara1991}. We will show a fundamental $\varepsilon$-regularity result that allows us to bound the infinity norm of all derivatives for some time if the energy is small on a certain scale. This result enables us to characterize the formation of a singularity in terms of concentrations of energy and allows us to  construct a blow-up profile at a
possible singularity. This solves one of the open problems listed by Zheng-Xu He in \cite{He2000}.

Ruling out blow-ups for planar curves, we will prove that the flow transforms every planar curve into a round circle.
\end{abstract}

\maketitle


\section{Introduction}

In their seminal  paper \cite{Freedman1994}, Freedman, He, and Wang suggested the study of the negative gradient flow of the M\"obius energy introduced by O'Hara in \cite{OHara1991}. For a closed curve $\gamma \in C^{0,1}(\mathbb R / l\mathbb Z, \mathbb R^n)$, $l >0$, this energy is given by
\begin{equation}
 E(\gamma) := \iint_{(\mathbb R / l\mathbb Z)^2} \left( \frac 1 {|\gamma(x)-\gamma(y)|^2} - \frac 1 {d_\gamma(x,y)^2}  \right) |\gamma'(x)|  \, |\gamma'(y)| \, dx \, dy
\end{equation}
where $d_\gamma(x,y)$ denotes the distance 
of the two points $\gamma(x),\gamma(y)$ along $\gamma$. Among many other things, Freedman, He, and Wang could show that curves of finite energy are tame and that the M\"obius energy can be minimized within every prime knot class \cite{Freedman1994}. Abrams et al. proved that the circle minimizes the energy among all closed curves \cite{Abrams2003}. It is an open problem whether these energies can be minimized within composite knot classes or not.

The evolution equation is governed by the law
\begin{equation} \label{eq:GFME}
 \partial_t \gamma = -\mathcal H \gamma
\end{equation}
where
\begin{equation*} \label{eq:gradientME}
 \mathcal H \gamma (x):= 2 \, p.v. \int_{- \frac l2}^{\frac l2}
 \left(2 \frac {P_{\gamma'}^\bot ( \gamma(x+w) - \gamma(x))}{|\gamma(x+w) - \gamma(x)|^2} - \kappa_\gamma (x) \right)
 \frac {|\gamma'(x+w)| dw}{|\gamma(x+w) - \gamma(x)|^2}
\end{equation*}
and $P^\bot_{\gamma'}w := w-\left\langle w, \frac {\gamma'}{|\gamma'|} \right\rangle \frac {\gamma'}{|\gamma'|}$ denotes the orthogonal projection onto the normal part along the curve $\gamma$ \cite[Lemma~6.1]{Freedman1994}.  Here, $p.v \int_{-\frac l2}^{\frac l2}$ denotes Cauchy's principal value, i.e. is an abbreviation for $\lim_{\varepsilon \downarrow 0} \int_{I_{l,\varepsilon}}$ where $I_{l, \varepsilon} = [-\frac l2 , \frac l2] \setminus (-\varepsilon, \varepsilon).$

If $\gamma$ is parameterized by arc-length this further reduces to 
\begin{equation} \label{eq:gradientArcLength}
 \mathcal H \gamma (x):= 2 \, p.v. \int_{-\frac l2}^{\frac l2}
 \left(2 \frac {P_{\gamma'}^\bot ( \gamma(x+w) - \gamma(x))}{|\gamma(x+w) - \gamma(x)|^2} - \gamma'' (x) \right)
 \frac  {dw}{|\gamma(x+w) - \gamma(x)|^2}.
\end{equation}

 Zheng-Xu He  observed that \eqref{eq:GFME}  is a quasilinear equation of third order and stated a short-time existence result for smooth curves using the Nash-Moser implicit function theorem \cite[Theorem~2.1]{He2000}. Using refined estimates, in \cite{Blatt2011c} we proved short-time existence for embedded $C^{2+\alpha}$-curves by Banach's fixed-point theorem. Furthermore, we have shown, using a {\L}ojasiewich-Simon gradient estimate, that local minimizers of the energy are attractive in the sense that there is a $C^{2+\alpha}$-neighborhood of initial data for which the flow exists for all time and converges to a local minimizer. Lin and Schwetlick \cite{Lin2010} considered the elastic energy plus some positive multiple of the M\"obius energy and the length. They could show long-time existence for the related negative gradient flow and convergence to critical points by essentially treating the flow as a perturbation of the elastic flow investigated in \cite{Dziuk2002}.

In this paper we derive an $\varepsilon$-regularity result for the evolution equation \eqref{eq:GFME} that will be essential in the analysis of the long-time behavior of the flow. As for the Willmore flow \cite{Kuwert2002} or the biharmonic and polyharmonic heat flow in the critical dimension \cite{Lamm2004, Gastel2006} a quantum of the energy has to concentrate whenever a singularity forms.

 For any measurable subset $A \subset \mathbb R^n$ we define the localized energy
\begin{equation} \label{eq:ExtLocalEnergy}
  E_A(\gamma) := \iint_{(\gamma^{-1}(A))^2} \left( \frac 1 {|\gamma(x) - \gamma(y)|} - \frac 1 {d_\gamma(x,y)}\right) |\gamma'(x)| |\gamma'(y)| dx dy.
\end{equation}

\newtheorem*{thmA}{Theorem~\ref{thm:regularity}}

\begin{thmA}[$\varepsilon$-regularity] 
  There are constants $\varepsilon_0>0$  and $C_k < \infty$, $k \in \mathbb N$, depending only on $n$ and $E(\gamma_0)$
 such that the following holds:
 Let $\gamma_t$, $t \in [0,T)$ be a maximal smooth solution of \eqref{eq:GFME} and let 
 $t_0 \in [0,T)$, $r>0$ be such that
 \begin{equation*}
  \sup_{x \in \mathbb R^n} E_{B_r(x)} (\gamma) \leq \varepsilon_0.
 \end{equation*}
 Then $T > t_0+ r^3$ and
 \begin{equation*}
  \| \partial_s^k \gamma_{t_0+r^3} \|_{L^\infty} \leq \frac{C_k}{(rt)^{\frac{k-1}3}} \quad \forall t \in (t_0,t_0 + r^3].
 \end{equation*}
\end{thmA}

Though the structure of this result is similar to many well-known $\varepsilon$-regularity results for critical evolution equations, due to the non-locality of the equation one has to develop new techniques in order to prove this theorem.  These techniques will certainly be applicable to other non-local geometric partial differential equations. The main stategy is to consider the evolution of localized energies and derive differential inequalities. Due to the non-locality of the equation however, non-local terms appear in these inequality which make it impossible to apply Gronwall's lemma. We will see that instead a "point-picking method" well help us out.

 As a first consequence of this result we prove the following concentration compactness alternative for the flow. 

\newtheorem*{thmB}{Theorem~\ref{blowup}}
\begin{thmB}[Characterization of singularities]
 Let $\gamma \in C^\infty([0,T)\times \mathbb R / \mathbb Z,\mathbb R^n)$ be a maximal smooth solution of \eqref{eq:GFME}. There is a constant $\varepsilon_0 >0$ depending only on $n$ and $E(\gamma_0)$ such that if $T<\infty$ there are times $t_k \uparrow T$, points $x_k \in \mathbb R^n$ and radii $r_k \downarrow 0$ with
 $$
  E_{B_{r_k}(x_k)}(\gamma_{t_k}) \geq \varepsilon_0
 $$
\end{thmB}

If a singularity occurs, then, by choosing the points $x_j$ in the last theorem more carefully, we can furthermore construct a so called \emph{blow-up profile}. It is simpler to formulate this theorem using the intrinsically defined local energies
\begin{multline*}
 E^{int}_{B_r(x_0)} (\gamma) \\:= \int_{d_\gamma(y,x_0) \leq r} \int_{d_\gamma(x,x_0) \leq r} \left( \frac 1 {|\gamma(x) - \gamma(y)|^2} - \frac 1 {d_\gamma(x,y)^2}\right) |\gamma'(x)| |\gamma'(y)| dx dy.
\end{multline*}

\newtheorem*{thmC}{Theorem~\ref{thm:ExistenceOfABlowupLimit}}
\begin{thmC} [Blow-up profile] 
There is an $\varepsilon_0 >0$ such that the following holds:
 Assume that $\gamma_t$ is a solution to \eqref{eq:GFME} that develops a singularity in finite time,
 i.e. $T<\infty$ and $r_j \rightarrow 0$. 
 Then there are points $x_j$ and times $t_j \rightarrow T$ such that
 $$
  E^{int}_{B_{r_j}(x_j)} (t_j) \geq \varepsilon_0.
 $$
 Let us now choose the points $x_j \in \mathbb R$ and times $t_j \in [0,T)$ such that
 \begin{equation*}
    \sup_{\tau \in [0,t_j], x \in \Gamma_\tau} E^{int}_{B_{r_j} (x)} (\gamma_{t_j})
    \leq E^{int}_{B_{r_j} (x)} (\gamma_{t_j}) = \varepsilon_0,
 \end{equation*}
and let $\tilde \gamma_j$ be re-parameterizations by arc-length of the rescaled and translated curves $$r_j^{-1} (\gamma_{t_j}-x_j)$$ such that
$\tilde \gamma_j(0) \in B_2(0)$. Then these curves 
sub-converge locally in $C^\infty$ to an embedded closed or open curve $\tilde \gamma_\infty: I \rightarrow \mathbb R^n$, $I = \mathbb R / l \mathbb Z$ or $I = \mathbb R$ resp.,  parameterized by arc-length. This curve satisfies
\begin{equation} \label{eq:EL}
 p.v \int_{\frac l2}^{\frac l2} \left( 2 \frac {P_\tau^\bot \left( \tilde \gamma(y)-\tilde \gamma(x) \right)}{| \gamma(y)- \gamma(x)|^2}
  - \kappa_\gamma (x) \right) \frac {dy}{|\gamma(y)-\gamma(x)|^2}  =0 \quad \quad \forall x \in  I, 
\end{equation}
and
\begin{equation*}
 E^{int}_{\overline B_1(0)} (\tilde \gamma_\infty)\geq \varepsilon_0.
\end{equation*}
\end{thmC}

This  solves problem 2 of the open problems list in He's article \cite{He2000}. In the last part of this paper, we deduce a geometric interpretation of the Euler-Lagrange equation of the M\"obius energy. In the case of co-dimension one, He could show that the only closed critical curves or the M\"obius energy are the circles. We will see that unfortunately the blow-up profiles are non-compact. Therefore we cannot apply this result of He in this context. Our new interpretation of the Euler-Lagrange equation allows us to show that the only planar solutions to the Euler-Lagrange equation \eqref{eq:EL} are straight lines and circles. Combining this result with a careful analysis of the asymptotic behavior of the flow, we can finally show

\newtheorem*{thmD}{Theorem~\ref{thm:GFMEplanar}}
\begin{thmD} [Planar curves]
 Let $\gamma_0 \subset \mathbb R^2$ be a closed smoothly embedded curve. Then the negative gradient flow of the M\"obius energy exists for all times and converges to a round circle as time goes to infinity.
\end{thmD}

Though from the topological point of view the case of planar curves is of no interest, the techniques that lead to this last result reduce the study of the flow to the study of compact and non-compact smooth solutions of the Euler-Lagrange equation \eqref{eq:EL} in the very intuitive geometric form \eqref{eq:ELGeometric}. Surprisingly, in the classification of planar blow-up profile this equation is only used in one point which gives hope that this geometric version of the equation might help to classify blow-up profiles in other situations.

\section{Preliminaries and Notation}

As for most of our estimates the precise algebraic form of the terms
does not matter, we will use the following notation to describe the 
essential structure of the terms.

For two Euclidean
vectors $v,w$, $v \ast w$ stands for a bilinear operator in $v$ and $w$
into another Euclidean vector space. For a regular curve $\gamma$, let $\partial_s = \frac {\partial_x}{|\gamma'|}$ denote the derivative with respect to arc length. For $\mu, \nu \in \mathbb N$,  a
regular curve $\gamma \in C^\infty (\mathbb R / \mathbb Z, \mathbb R^n)$
and a function $f: \mathbb R / \mathbb Z \rightarrow \mathbb R^k$,
let $P^{\mu}_\nu(f)$ be a linear combination of terms of the form
$\partial_s^{j_1} f \ast \cdots \ast \partial_s^{j,\nu} f$, $j_1 +
\cdots + j_\nu =\mu$. Furthermore, given a second function $g: \mathbb R /
\mathbb Z \rightarrow \mathbb R^k$ the expression
$P^{\mu}_{\nu}(g,f)$ denotes a linear combination of terms of the
form $\partial_s^{j_1} g \ast \partial_s^{j_2} f \ast  \partial_s^{j_3} f \ast  \cdots \ast \partial_s^{j,\nu} f$, $j_1 +
\cdots + j_\nu =\mu$.

\subsection{Decomposition of the gradient and the operator $Q$}

We will always assume that our curve is parameterized by arc length at the fixed time $t$ we currently consider.
Whenever we have to estimate $\mathcal H$ we will write it as
\begin{equation} \label{eq:TildeH}
 P_{\gamma'}^\bot \tilde {\mathcal H}
\end{equation}
where
$$
 \tilde {\mathcal H} \gamma(x)= 2 \,  p.v. \int_{- \frac l2}^{\frac l2}
 \left(2\frac{  \gamma(u+w) - \gamma(u) - w \gamma'(u)}{|\gamma(u+w) - \gamma(u)|^2} - \gamma'' (x) \right)
 \frac  {dw}{|\gamma(u+w) - \gamma(u)|^2}
$$
and decompose
\begin{equation} \label{eq:DecomositionH}
 \tilde {\mathcal H} \gamma = Q \gamma + R_1 \gamma + R_2 \gamma = Q \gamma + R \gamma
\end{equation}
where
 \begin{align*}
  Q \gamma(x) &= 2 \,  p.v. \int_{- \frac l2}^{\frac l2}\left(2  \frac {\gamma(x+w) - \gamma(x) - w \gamma'(x)} {w^4} - \frac {\kappa (x)} {|w|^2} \right) dw   
  \\ & = 4 \,  p.v. \int_{- \frac l2}^{\frac l2} \int_0^1 (1-s)\frac{ \kappa(x+sw) - \kappa (x)} {|w|^2} ds dw \\ & = \tilde Q \kappa  (x),\\
  R_1 \gamma (x) &=  4 \int_{- \frac l2}^{\frac l2}   (\gamma(x+w) - \gamma(x) - w \gamma'(x)) \left(\frac 1 {|\gamma(x+w) - \gamma(x) |^4} -  \frac 1  {w^4}  \right) dw ,\\
  R_2 \gamma (x) &=:  2 \int_{- \frac l2}^{\frac l2}   \kappa(x) \left( \frac 1  {w^2} - \frac 1 {|\gamma(x+w) - \gamma(x)|^2} \right) dw.
 \end{align*}
He observed that the operator $Q$ can be written as a multiple of the fractional Laplacian $(-\Delta)^{\frac 32 }$ plus an operator of order $2$ \cite{He2000}. Let us state the consequences of his result for the operator $\tilde Q$ of order $1$:

\begin{lemma} \label{lem:FourierCoefficientsQ}
 For every smooth function $f \in C^\infty(\mathbb R / l\mathbb Z, \mathbb R^n)$ we have
 \begin{equation*}
  \tilde Q f = \frac 1 l \sum_{k \in \mathbb Z} \frac{\lambda_k} k \hat f(k)
 \end{equation*}
 where $\hat f(k)$ denotes the $k$-th Fourier coefficient and $\lambda_k = \frac \pi 3 + O(\frac 1 k)$. Hence, for $l \geq 1$ we have
 \begin{equation*}
  \left|\frac {\pi^2} 9\|f'\|^2_{L^2} -\|\tilde Qf\|^2_{L^2}\right| \leq C \|f\|^2_{L^2}.
 \end{equation*}

\end{lemma}

Let us add another useful identity for the operator $Q$ to the two identities we already have given above.
For smooth $f,g$ we observe, using first partial integration and then discrete partial integration,  
\begin{align*}
  & \quad  \int_{\mathbb R / l \mathbb Z}p.v. \int_{-\frac l2}^{\frac l2}  \int_0^1 (1-s)\frac{ f''(x+sw) - f'' (x)} {|w|^2}  ds dw g(x) dx  
  \\ & =  \int_{\mathbb R / l \mathbb Z}p.v. \int_{-\frac l2}^{\frac l2}\int_0^1 (1-s)\frac{ f'(x+sw) - f' (x)} {|w|^2} g'(x) dw dx
  \\ & =  \frac 1 2 \Bigg(   \int_{\mathbb R / l \mathbb Z}p.v. \int_{-\frac l2}^{\frac l2} \int_0^1 (1-s)\frac{ f'(x+sw) - f' (x)} {|w|^2} g'(x) ds dw dx 
   \\ & \quad \quad \quad -  \int_{\mathbb R / l \mathbb Z}p.v. \int_{-\frac l2}^{\frac l2}\int_0^1 (1-s)\frac{ f'(x+sw) - f' (x)} {|w|^2} g'(x+sw) dw dx  \Bigg)
   \\ & = \frac 12\int_{\mathbb R / l \mathbb Z}p.v. \int_{-\frac l2}^{\frac l2} \int_{0}^1 (1-s) \frac {(f'(x+sw)- f'(x))( g'(x+sw) - g'(x))}{w^2} ds dw dx.
\end{align*}
Hence, as we do not need the principal value to make sense of the last expression we have
\begin{multline} \label{eq:NewFormulaForQ}
 \int_{\mathbb R / l \mathbb Z }  \langle Q f,g \rangle ds \\ = 2 \int_{\mathbb R / l\mathbb Z} \int_{-\frac l2} ^{\frac l2} \int_{0}^1 (1-s) \frac {(f'(x+sw)- f'(x))( g'(x+sw) - g'(x))}{w^2} ds dw dx.
\end{multline}

\subsection{Coercivity of the M\"obius energy and Bi-Lipschitz estimates}

Of fundamental importance in the following is the deep connection between the M\"obius energy and fractional Sobolev spaces observed in \cite{Blatt2012} which was sharpened in \cite[Theorem~3.2]{Blatt2016}. We showed there that the M\"obius energy of an embedded curve parameterized by arc length is finite if and only if the curve is of class $w^{\frac 3 2,2}$. More precisely, we have

\begin{theorem} [Characterization of finite energy curves] \label{thm:FiniteEnergy}
  Let $\gamma \in C^{1}(\mathbb R / l \mathbb Z , \mathbb R^n )$ be a curve parameterized 
by arc length. Then the energy $E(\gamma)$ is finite if and only if $\gamma \in W^{\frac 3 2,2}$. Moreover there exists a constant $C< \infty$ not depending on $\gamma$ such that
\begin{equation} \label{eq:Coer}
	\|\gamma'\|_{W^{\frac {3}{2},2}} \leq C \left(E (\gamma) \right).
\end{equation}
\end{theorem}

So especially, for a solution of the gradient flow \eqref{eq:GFME} the $W^{\frac 3 2,2}$-norm of the gradient after reparametrizing the curve by arc-length is uniformly bounded in time.
An essential ingredient of the proof of the theorem above and the analysis in this article is the following bi-Lipschitz estimates for curves of finite energy of O'Hara \cite{OHara1991}. This bi-Lipschitz constant is also  well-known under the term \emph{Gromov distortion}.

\begin{lemma} [Bi-Lipschitz estimate]\label{lem:BiLipschitz}
For an injective curve $\gamma \in W^{3/2,2}(\mathbb R /l \mathbb Z, \mathbb R^n)$ we get the following bound of the Gromov distortion
\begin{equation*}
  \dist = \dist(\gamma)=\sup_{x \not=y} \frac {d_\gamma(x,y)}{|\gamma(x)-\gamma(y)|} \leq 18 e^{\frac {E(\gamma)} 4} .
\end{equation*}
\end{lemma}
If $\gamma$ is parameterized by arc-length, we obtain
\begin{equation}\label{eq:bilipschitz}
 \frac {|w|}{|\gamma(x+w) - \gamma(x)|} \leq  18 e^{\frac {E(\gamma)} 4} , \quad \forall x,w \in \mathbb R,  |w|\leq \frac l2.
\end{equation}

Let us sketch how this bi-Lipschitz estimate was used in \cite{Blatt2012} to prove Theorem~\ref{thm:FiniteEnergy}.
For a curve $\gamma \in W^{3/2,2}(\mathbb R / l\mathbb Z , \mathbb R^n)$ parameterized by arc-length, $x \in \mathbb R / l\mathbb Z$
and $0<|w|<l/2 $, we deduce using this bi-Lipschitz estimate the following estimate for the integrand of the energy
\begin{equation} \label{eq:EstEnergyIntegrand}
\begin{aligned}
  \frac 1 {|\gamma(x+w)| - \gamma(x)|^2} - \frac 1 {w^2} & =  \frac {w^2}{|\gamma(x+w)-\gamma(x)|^2} \frac {1- \frac {|\gamma(x+w)-\gamma(x)|^2} {w^2}}{|w^2|} \\
  &\leq \frac \beta 2\int_0^1 \int_0^1 \frac {|\gamma'(x+s_1 w) - \gamma'(x+s_2 w)|^2}{w^2} ds_1 ds_2
  \\ &\leq 2 \beta \int_0^1 \frac{|\gamma(x+s w) - \gamma(x)|^2}{w^2} ds.
 \end{aligned}
\end{equation}
One then derives the statement of Theorem~\ref{thm:FiniteEnergy} by basically integrating this inequality over all $x$ and $w$.

More generally, we get for $\alpha \geq 0$ and using that the function $ x \rightarrow \frac {1-x^{2+\alpha}}{1-x^2}$
is locally bounded on $(0, \infty)$ that
\begin{equation} \label{eq:EstGeneralEnergyIntegrand}
\begin{aligned}
  \frac {|w|^\alpha}  {|\gamma(x+w)| - \gamma(x)|^{\alpha+2}} - \frac 1{w^{2}} & =  \frac {|w|^{2+\alpha}}{|\gamma(x+w)-\gamma(x)|^{2+\alpha}} \frac {1- \frac {|\gamma(x+w)-\gamma(x)|^{2+\alpha}} {|w|^{2+\alpha}}}{|w^2|} \\
  & \leq C\frac {1- \frac {|\gamma(x+w)-\gamma(x)|^{2}} {|w|^{2}}}{|w^2|} \\
  &\leq C\int_0^1 \int_0^1 \frac {|\gamma'(x+s_1 w) - \gamma'(x+s_2 w)|^2}{w^2} ds_1 ds_2
  \\ & \leq  C\int_0^1 \frac{|\gamma(x+s w) - \gamma(x)|^2}{w^2} ds
 \end{aligned}
 \end{equation}
  where the constant $C$ depends only on an upper bound for $\beta$ like $E(\gamma)$ and $\alpha$.
 Furthermore, we have
 \begin{align*}
  \frac 1{|\gamma(x+w) - \gamma(x) |^2} - \frac 1 {w^2} = \frac {\frac {w^2}{|\gamma(x+w) - \gamma(x)|^2} - 1}{w^2} \leq \frac {18^2 e^{\frac {E(\gamma)}{2}} -1 } {w^2}.
 \end{align*}

So we get the rough estimate
 \begin{equation}
 \begin{aligned}
  \int_{B_r(x)} \int_{\frac l 2 \geq |w| \geq \Lambda r} \left( \frac 1{|\gamma(x+w) - \gamma(x) |^2} - \frac 1 {w^2}\right) dw dx
  &\leq C(\beta)\int_{\Lambda r}^{\infty} \frac {dw}{w^2} \\
  & \leq C(\beta)
 \end{aligned}
 \end{equation}

\subsection{Fractional Sobolev spaces and Besov spaces} In our calculation, fractional Sobolev spaces as well as Besov spaces naturally appear. 
For an introduction to Besov spaces we refer to the monographs \cite{Triebel1983} and \cite{Triebel1983}.
Let $f\in L^1(\mathbb R / l \mathbb Z)$. For $s\in(0,1)$ and $p,q \in[1,\infty)$ and for open subsets $\Omega \subset \mathbb R /l \mathbb Z$ we also consider the \emph{Besov type} seminorm
\begin{equation}\label{eq:Wsemi}
  |f|_{B^s_{p,q}(B_R(x))} := \left( \int_{B_R(x)}  \frac{\left( \int_{-\frac R2} ^{\frac R 2} |f'(u+w)-f'(u)|^pdu \right)^{\frac q p}}{|w|^{1+q s}}  dw) \right) ^{\frac 1q}.
\end{equation}
It is shown in the appendix that
\begin{equation*}
 |f|_{B^s_{p,q}(B_R(x))} \leq C \|f\|_{B^{s}_{p,q}(B_{2R}(x))}
\end{equation*}
and
\begin{equation*}
 \|f\|_{B^{s}_{p,q}(B_{R}(x))} \leq C \left( |f|_{B^s_{p,q}(B_{2R}(x)) + \|f\|_{L^p(B_{2R}(x))} }\right).
\end{equation*}

\section{An \protect{$\varepsilon$}-regularity  result}

In this section we prove the main result of this article, an $\varepsilon$-regularity result for the flow \eqref{eq:GFME}:

\begin{theorem}[$\varepsilon$-regularity] \label{thm:regularity}
 There are constants $\varepsilon>0$  and $C_k < \infty$, $k \in \mathbb N$, depending only on $n$ and $E(\gamma_0)$
 such that the following holds:
 Let $\gamma_t$, $t \in [0,T)$ be a maximal smooth solution of \eqref{eq:GFME} and let
 $t_0 \in [0,T)$, $r$ be such that
 \begin{equation} \label{eq:concentration}
  \sup_{x \in \mathbb R^n} E_{B_r(x)} (\gamma_{t_0}) \leq \varepsilon.
 \end{equation}
 Then $T > t_0+ r^3$ and 
 \begin{equation*}
  \| \partial_s^k \gamma_{t_0+r^3} \|_{L^\infty} \leq \frac{C_k}{(rt)^{\frac{k-1}3}} \quad \forall t \in (t_0,t_0 + r^3].
 \end{equation*}
\end{theorem}

\begin{remark}
Note that the assumptions in the theorem are highly non-local. It is a very interesting and challenging question whether one can prove a local version of this regularity theorem. 
\end{remark}

Clearly, one only has to prove Theorem~\ref{thm:regularity} only for the special case $t_0=0$ and $r=1$. Scaling and translation in time then gives the full statement.

We will prove Theorem~\ref{thm:regularity} in three steps using energy estimates for this special case. First we control the energy within a ball of radius $1$ at later times, before we estimate the elastic energy, i.e. the $L^2$-norm of the curvature. In a last step we will then bound higher order energies.
The general strategy will always be to derive evolution equations for the quantities and use the quasilinear structure together with interpolation estimates in order to derive differential inequalities (cf. Lemmata~\ref{lem:APrioriEstimateI}, \ref{lem:inequalityHigherEnergyCritical},
and \ref{lem:inequalityHigherEnergy}).

Due to the non-local structure of the inequalities, though we start with local quantities these differential inequalities are also non-local which makes the usual application of the Gronwall's lemma impossible. A kind of point picking method will help us there.

\subsection{Estimates for the energy density} \label{subsec:EstimatesEnergyDensity}

Let us fix a radial cutoff function $\phi(x)= \phi(|x|) \in C_c^\infty (\mathbb R^n)$
such that
\begin{equation*}
\chi_{B_1(0)} \leq \phi \leq \chi_{B_2(0)}.
\end{equation*} For $x_0 \in \mathbb R^n$ we set $\phi_{x_0}(x):= \phi(x-x_0)$ and define the localized energy
\begin{equation} \label{eq:ExtLocalEnergyPhi}
  E^{\phi_{x_0}}(\gamma) := \iint_{(\mathbb R / \mathbb Z)^2} \left( \frac 1 {|\gamma(x) - \gamma(y)|^2} - \frac 1 {d_\gamma(x,y)^2}\right) |\gamma'(x)| |\gamma'(y)|  \phi_{x_0} (\gamma(x))dx dy
\end{equation}

A straight forward calculation leads to the following evolution equation for $E^\phi$. We leave the proof to the reader.

\begin{lemma} [Evolution equation for local density] \label{lem:EvolutionOfDensity}
Let $\gamma_t$ be parameterized by arc length and $\frac d {dt} \gamma_t = V$ be orthogonal to $\gamma_t$. Then we have
\begin{align*}
  &\frac d {dt} E^\phi (\gamma_t) \\ & =  2 \, p.v.  \int_{-\frac l2 }^{\frac l2 } \int_{I_{l,\varepsilon}}   \left\langle 2 \frac { \gamma(x+w) -  \gamma(x) }{| \gamma(x+w) -  \gamma(x)|^4} - \frac {\kappa_{ \gamma}(x)}{| \gamma(x+w)-  \gamma(x)|^2} , V(x) \right\rangle \phi ( \gamma(x)) dw dx
  \\ & \quad + 2 \int_{\mathbb R / l \mathbb Z} \int_{-\frac l 2}^{\frac l2 } \frac {\langle  \gamma(x+w) -  \gamma(x) - w \gamma'(x)- \frac 1 2 |\gamma(x+w) - \gamma(x)|^2 \kappa (x) , V(x) \rangle}{|\gamma(x+w) - \gamma(x)|^4} \\ & \quad \quad \quad \quad \quad \quad \quad \quad  \left( \phi(\gamma(x+w)) - \phi(\gamma(x))\right)
  dw dx 
  \\ & + \int_{\mathbb R / l \mathbb Z} \int_{-\frac l 2}^{\frac l2 } \frac {\langle \kappa(x), V(x) \rangle}{|w|^2} \left( \phi(\gamma(x+w) + \phi(\gamma(x)) -2 \int_0^1\phi(\gamma(x+\tau w)) d\tau)\right) dw dx
  \\ &  + \int_{\mathbb R / l \mathbb Z} \int_{-\frac l 2}^{\frac l2 }   \left( \frac 1 {| \gamma(x+w) -  \gamma(x)|^2} - \frac 1 {|w|^2} \right)\langle V(x), \nabla \phi(\gamma(x))\rangle dwdx 
  \\ & =: I_1+I_2+I_3+I_4
  \end{align*}
where $l$ is the length of $\gamma_t$. 
\end{lemma}

In the rest of this section we estimate these terms for the case
\begin{equation*}
 V = \mathcal H \gamma.
\end{equation*}
To make the calculations and formulas as simple as possible, we always assume that the curve $\gamma_t$ is parameterized by arc-length at the current time $t$. We will use the intrinsically defined quantities
\begin{align*}
 M_{\frac 3 2} & = M_{\frac 32 } (t) = \sup_{x\in \mathbb R / l \mathbb Z}  \int^{\frac 1 2}_{-\frac 1 2}
 \int^{\frac l 2}_{-\frac l 2} \frac {|\gamma'(x+w)-\gamma'(x)|^2} {w^2} dw dx,
 \\ S_3(x) &= S_3(x,t) = \|\partial_s^3 \gamma_t\|^2_{L^2(B_\Lambda(x))} + \sum_{j=1}^\infty \frac{\|\partial_s^3\gamma_t\|^2_{L^2(B_{\Lambda + j}(x)
 \setminus B_{\Lambda + (j-1)}(x))} }{(\frac \Lambda 2 + j)^2} 
 \intertext{and}
  \tilde S_3 (x)& = \tilde S_3(x,t) = \|\mathcal H \gamma_t(x) \|^2_{L^2(B_\Lambda(0))} + \sum_{j=1}^\infty \frac{\|\mathcal H \gamma_t \|^2_{L^2(B_{\Lambda + j}(x))} }{(\frac \Lambda 2 + j)^2}
\end{align*}
for $\Lambda = 1000 \cdot 18 e^{\frac{E(\gamma_0)} 4} < \infty$. Note, that due to Lemma~\ref{lem:BiLipschitz} the quantity $18 e^{\frac {E(\gamma_0)} 4}$ bounds the Gromov distortion of the curves $\gamma_t$ for all $t$. Hence, $\Lambda$ is large compared with the Gromov distotion of $\gamma$.

To prevent complicated terms in the estimates, we will assume throughout this section that
\begin{equation*}
 M_{\frac 3 2} \leq 1.
\end{equation*}
Furthermore, we will assume that $\gamma(0) \in B_2(0)$ to get some preliminary estimates in terms of the intrinsically defined quantities above.
In the final differential inequality we will use the extrinsic quantity
\begin{equation*}
 S_3^{ext}(x,t) := \|\mathcal H \gamma(x) \|^2_{L^2(\gamma^{-1}(B_1(0)))} + \sum_{j=1}^\infty \frac{\|\mathcal H \gamma \|^2_{L^2(\gamma^{-1}(B_{j+1}(x)\setminus B_{j}(x) )} }{ j^2}
\end{equation*}
in place of $S_3(0)$.

Let us start with the following easy, but useful lemma that will help us to control the part of the integrals defining $I_i$, $i=1, \ldots ,4$ for the case that $|w|$ is large.
\begin{lemma} \label{lem:EasyEstimate}
 For all $s \in [0,1]$, $p \in [1, \infty)$ and $x \in B_\dist (0)$ we have
\begin{equation*}
 \int_{|w| \geq \Lambda} \frac {|f(x+sw)|^p} {w^2}dw \leq C \|f\|^p_{L^p (B_\Lambda(0))} +
 \sum_{j \in \mathbb N} \frac{\|f\|^p_{L^p (B_{\Lambda+j}(0) - B_{\Lambda +j-1}(0))}}{(\Lambda +j)^2}
\end{equation*}
\end{lemma}
\begin{proof}
 The statement obviously holds for $s=0$. For $s>0$ we get substituting $\tilde w = sw$
\begin{align*}
 \int_{|w| \geq \Lambda} \frac {|f(x+sw)|^p} {w^2}dw & = s \int_{\frac \Lambda 2 \geq |\tilde w|\geq s \frac \Lambda 2}  \frac {|f(x+\tilde w)|^p} {\tilde w^2}d\tilde w + s \int_{|\tilde w|\geq \frac \Lambda 2 }  \frac {|f(x+\tilde w)|^p} {\tilde w^2}d\tilde w \\
 & \leq C \|f\|^p_{L^p(B_\Lambda)} + s \int_{|\tilde w|\geq \frac \Lambda 4 }  \frac {|f(\tilde w)|^p} {\tilde w^2}d\tilde w \\
 \\ & \leq  C \|f\|^p_{L^p (B_\Lambda)} +
 \sum_{j \in \mathbb N} \frac{\|f\|^p_{L^p (B_{\Lambda+j}(0) \setminus B_{\Lambda +j-1}(0))}}{(\Lambda +j)^2}.
\end{align*}

\end{proof}

We start with estimating the term $I_1$, which contains the terms of highest order. The guideline for estimating the remainder terms will be throughout this section to distinguish between areas where $|w|$ is small and where $|w|$ is big. Combining this idea with the commutator estimates and interpolation inequalities in the appendix (cf. Lemma~\ref{lem:Commutator} and Lemma~\ref{lem:Interpolation}) we obtain the desired estimates. 
 
\begin{lemma} [Estimate for $I_1$] \label{lem:PrincipalPart} Let $ \gamma$  be parameterized by arc-length, $M_{\frac 32 } \leq 1$, and $\gamma(0) \in B_2(0).$ Then
 there is a constant $\alpha >0$ 
 \begin{equation*}
  I_1 = - \int_{\mathbb R / l \mathbb Z} |\mathcal H \gamma|^2 \phi( \gamma(x)) dx =  - \int_{\mathbb R/ \mathbb Z} | Q \gamma(x)|^2 \phi(\gamma(x)) dx   + R_I,
 \end{equation*}
 where for all $\varepsilon >0$
 \begin{equation*}
  R_I \leq  \left(C_\varepsilon M^\alpha_{\frac 32} + \varepsilon \right) S_3(0)  + C_\varepsilon
 \end{equation*}
 for some $C_\varepsilon < \infty$.
\end{lemma}

\begin{proof}
 
 We have
 \begin{equation*}
  \mathcal H \gamma (x) = P^\bot_{\gamma'(x)} \left(Q \gamma(x) + R_1 \gamma(x) + R_2 \gamma(x) \right)
 \end{equation*}
 where
 \begin{align}
  Q \gamma(x) &= 2 \lim_{\varepsilon \downarrow 0} \int_{I_{l, \varepsilon}} \left(2  \frac {\gamma(x+w) - \gamma(x) - w \gamma'(x)} {w^4} - \frac {\kappa (x)} {|w|^2} \right) dw  \label{eq:FormulaForQ} 
  \\ & = 4 \lim_{\varepsilon \downarrow 0} \int_{I_{l, \varepsilon}} \int_0^1 (1-s)\frac{ \kappa(x+sw) - \kappa (x)} {|w|^2}  dw \nonumber \\ & = \tilde Q  \kappa  (x), \nonumber\\
  R_1 \gamma (x) &=  4 \int_{I_l}   (\gamma(x+w) - \gamma(x) - w \gamma'(x)) \left(\frac 1 {|\gamma(x+w) - \gamma(x) |^4} -  \frac 1  {w^4}  \right) dw , \nonumber \\ \intertext{and}
  R_2 \gamma (x) &=  2 \int_{I_l}   \kappa(x) \left( \frac 1  {w^2} - \frac 1 {|\gamma(x+w) - \gamma(x)|^2} \right) dw. \nonumber
 \end{align}
The bi-Lipschitz estimate together with $\gamma(0) \in B_2(0)$ tells us that $\phi(\gamma(x)) =0 $ for all $x\notin B_{2 \beta}(0)$. This yields
\begin{equation} \label{eq:PartsOfQ}
\begin{aligned}
  -\int_{\mathbb R / l \mathbb Z} |P_\gamma^\bot &(Q\gamma(x))|^2 \phi(\gamma(x)) dx
 \\&= -\int_{\mathbb R / l \mathbb Z} |Q\gamma(x)|^2 \phi(\gamma(x)) dx
 + \int_{\mathbb R / l \mathbb Z} |\langle Q\gamma(x), \gamma'\rangle|^2 \phi(\gamma(x)) dx
 \\ & \leq -\int_{\mathbb R / l \mathbb Z} |Q\gamma(x)|^2 \phi(\gamma(x)) dx
 + \int_{-2 \dist }^{2\dist} |\langle Q\gamma(x), \gamma'\rangle|^2  dx.
\end{aligned}
\end{equation}
Using that $\langle \kappa, \gamma'\rangle =0$ and that $\tilde Q$ is a linear operator, we get
\begin{align*}
 |\langle Q \gamma (x), \gamma' \rangle| & = |\langle \tilde Q \kappa (x), \gamma'\rangle|
 = |\langle \tilde Q \kappa (x), \gamma'\rangle - \tilde Q [\langle \kappa, \gamma'\rangle] (x) | 
 \\ & = |\sum_{i=1}^n \left( \tilde Q [\kappa_i] (x) \gamma'_i(x)- \tilde Q [ \kappa_i \gamma'_i] (x)\right)|.
\end{align*}
Hence, applying first the commutator estimate (Lemma~\ref{lem:Commutator}) 
and then the interpolation estimates (Lemma~\ref{lem:Interpolation})  we obtain
\begin{equation} \label{eq:QTangentialPart}
\begin{aligned}
 \|\langle &Q \gamma, \gamma'\rangle \|_{L^2(B_{2 \dist }(0))} \\&\leq C \left( \|\kappa\|_{B^{\frac 1 2 }_{4,2} (B_{\Lambda }(0))} \, \|\gamma'\|_{B^{\frac 1 2 }_{4,2}(B_{\Lambda}(0))} + \|\kappa\|_{L^2(B_{\Lambda}(0))} \,(\|\gamma'\|_{C^{0,1}(B_{\Lambda}(0))} +1) \right)
 \\
 &\leq C\left( M_{3/2}^{\frac 1 2} S^{\frac 1 2}_3(0) + M_{\frac 3 2} \right)  \leq C\left( M_{3/2}^{\frac 1 2} S^{\frac 1 2}_3(0) + 1\right).
\end{aligned}
\end{equation}

Using Taylor's theorem and \eqref{eq:EstEnergyIntegrand}, we get 
\begin{align*}
 |R_1 \gamma (x)| & = 4 \left | \int_{- \frac l2 }^{\frac l2} \int_0^1 (1-s) \kappa (x+sw) \left( \frac {w^2}{|\gamma(x+w) - \gamma(x)|^2} - \frac 1 {w^2} \right) ds dw \right| \\
 &\leq C \int_{I_l} \iiint_{[0,1]^3} |\kappa (x+s_1 w )| \frac {|\gamma'(x+s_2 w) - \gamma'(x+s_3 w)|^2}{|w|^2} ds_1 ds_2 ds_3  dw
 \\ &\leq C \int_{I_l} \iint_{[0,1]^2} \left| \kappa (x+s_1 w ) \right| \frac {|\gamma'(x+s_2 w) - \gamma'(x)|^2}{|w|^2} ds_1 ds_2 dw
 \\ &=  C (R_{11} \gamma (x) + R_{12}\gamma(x)),
\end{align*}
where 
\begin{align*}
 R_{11} \gamma (x) &= \int_{|w| \leq \frac  \Lambda 2} \iint_{[0,1]^2} g_{w,s_1,s_2}(x) ds_1 ds_2 dw \\
 R_{12} \gamma (x) &= \int_{\frac l 2 \geq |w| \geq \frac \Lambda 2} \iint_{[0,1]^2} g_{w,s_1,s_2}(x) ds_1 ds_2 dw \\
 \intertext{and}
 g_{w,s_1,s_2}(x) &:= |\kappa(x+s_1 w)| \, \frac { |\gamma'(x+s_2 w )-\gamma'(x)|^2}{|w|^2}.
\end{align*}
Since
\begin{align*}
 \int_{B_{2\dist}(0)} |g_{w,s_1,s_2}(x)|^2 dx
 \leq \|\kappa\|^2_{L^4(B_{\Lambda}(0))}
 \frac {\|\gamma'(\cdot +s_2w)- \gamma'\|^4_{L^8(B_{ 2\dist}(0))}}{|w|^2}
\end{align*}
we get
\begin{align*}
 \|R_{11}  & \gamma(x)  \|_{L^2(B_{2 \beta} (0))} \\ &\leq C \|\kappa\|_{L^4(B_{\Lambda}(0))}  \int_{|w| \leq \Lambda/2 } \iint_{0}^1 \frac {\|\gamma'(\cdot +s_2w)- \gamma'\|^2_{L^8(B_{2 \dist}(0))}}{|w|^2} ds dw
 \\ & \leq C  \|\kappa\|_{L^4(B_{\Lambda}(0))}  \|\gamma' \|^2_{B^{\frac 1 2}_{8,2}(B_{\Lambda} (0))}.
\end{align*}
Furthermore, since $|\gamma'| \equiv 1$ we get by Cauchy's inequality and Lemma~\ref{lem:EasyEstimate}
\begin{align*}
 |R_{12} \gamma (x)| &\leq 4 \int_{\frac l 2 \geq |w|\geq \Lambda /2} \int_0^1 \frac {|\kappa(x+s w)|} {|w|^2} ds dw \\
  & \leq C \left( \int_0^1  \int_{\frac l2 \geq |w| \geq \frac \Lambda 2} \frac { |\kappa (x+sw)|^2}{|w|^2} dw + 1 \right)
  \\ & \leq C \left( \|\kappa\|^2_{L^2(B_\Lambda(0))} + \sum_{j=1}^\infty \frac { \|\kappa\|^2_{L^2(B_{\Lambda +j} - B_{\Lambda + j -1})} }{(\Lambda + j)^2} + 1\right).
\end{align*}
Thus
\begin{align*}
 \|R_{12}  \gamma(x) \|_{L^2(B_{2\dist}(0))} &\leq   C \left( \|\kappa\|^2_{L^2(B_\Lambda(0))} + \sum_{j=1}^\infty \frac { \|\kappa\|^2_{L^2(B_{\Lambda +j} - B_{\Lambda + j -1})} }{(\Lambda + j)^2} + 1 \right).
\end{align*}
Together with the interpolation inequalities from Lemma~\ref{lem:Interpolation} this leads to
\begin{align*}
   &\|R_1 \gamma \|^2_{L^2(B_{2\dist} (0))} \\ &\leq C \left(  \|\kappa\|_{L^4(B_{\Lambda}(0))}^2  \| \gamma' \|^4_{B^{\frac 1 2}_{8,2}(B_\Lambda (0))} +  \|\kappa\|^2_{L^2(B_\Lambda(0))} + \sum_{j=1}^\infty \frac { \|\kappa\|^2_{L^2(B_{\Lambda +j} - B_{\Lambda + j -1})} }{(\Lambda + j)^2} + 1 \right)
  \\
  &\leq C \left(  M^{2}_{3/2} S_3+ \|\partial_s^3 \gamma\|^{\frac 43}_{L^2 (B_\Lambda)} 
  + \sum_{j=1}^\infty \frac {\|\partial_s^3 \gamma\|^{\frac 43}_{L^2 (B_{\Lambda+j}- B_{\Lambda +j -1})}}{(\Lambda +j)^2} +1 \right)
  \\ & \leq \left(C M^2_{3/2} + \varepsilon \right) S_3 + C_\varepsilon,
\end{align*}
where we have used the Cauchy inequality in the last step.

In the same way, one deals with the term $R_2$ to get
\begin{equation} \label{eq:R}
 \|P^\bot_{\gamma'}R\|^2_{L^2(B_{2\dist}(0))} \leq \|R\|^2_{B_{2\dist}(0)} \leq  (C M^3_{\frac 3 2} + \varepsilon ) S_3+ C_\varepsilon.
\end{equation}
From \eqref{eq:PartsOfQ}, \eqref{eq:QTangentialPart}, and \eqref{eq:R} the assertion follows.
\end{proof}

\begin{lemma} [Estimate for $I_2$] \label{lem:I2}
 Let $M_{\frac 32 } \leq 1$ and $\gamma(0) \in B_2(0).$ For all $\varepsilon >0$
 \begin{equation*}
  |I_2|  \leq  \varepsilon (S_3 + \tilde S_3) + C_\varepsilon
 \end{equation*}
 for some $C_\varepsilon < \infty$ depending only on $\varepsilon$ and $E(\gamma_0).$
\end{lemma}

\begin{proof}
 We decompose
 \begin{align*}
  I_2 &= 2 \int_{\mathbb R / l \mathbb Z} \int_{-\frac l 2}^{\frac l2 } \frac {\langle  \gamma(x+w) -  \gamma(x) - w \gamma'(x)- \frac 1 2 |\gamma(x+w) - \gamma(x)|^2 \kappa (x) , V(x) \rangle}{|\gamma(x+w) - \gamma(x)|^4}  \\ & \quad \quad \quad \quad \quad \quad \quad \quad \quad \quad \quad  \left( \phi(\gamma(x+w)) - \phi(\gamma(x))\right)
  dw dx
  \\ & =  2 \int_{\mathbb R / l \mathbb Z} \int_{-\frac l 2}^{\frac l2 } \frac {\langle  \gamma(x+w) -  \gamma(x) - w \gamma'(x)- \frac 1 2 |w|^2 \kappa (x) , V(x) \rangle}{|\gamma(x+w) - \gamma(x)|^4}  \\ & \quad \quad \quad \quad \quad \quad \quad \quad \quad \quad \quad \left( \phi(\gamma(x+w)) - \phi(\gamma(x))\right) dw dx
  \\ &\quad - \frac 12  \int_{\mathbb R / l \mathbb Z} \int_{-\frac l2}^{\frac l2 } \frac {|\gamma(x+w)-\gamma(x)|^2  - |w|^2}{|\gamma(x+w) - \gamma(x)|^4} \langle \kappa (x) , V(x) \rangle \left( \phi(\gamma(x+w)) - \phi(\gamma(x))\right) dw dx
  \\ & =: I_{21} + I_{22}.
 \end{align*}
Using the bi-Lipschitz estimate \eqref{eq:bilipschitz} and Taylor's approximation up to the first order, we get
 \begin{align*}
  I_{21} &\leq C \int_{\mathbb R / l\mathbb Z} \int_{-\frac l2}^{\frac l2} \int_0^1 \frac {|\kappa (x+sw) - \kappa (x)|}.
  {w^2} |V(x)| |\phi(\gamma(x+w) - \phi(\gamma(x))| ds dw dx
\end{align*}
Observing that $\phi(\gamma(x+w)) - \phi(\gamma(x)) = 0 $ if both $|x|, |x+w| \geq 2 \dist$, this can be estimated by
\begin{align*}
  I_{21}\leq & \quad  C  \int_{B_{\Lambda/2} (0)} \int_{-\frac l2 }^{\frac l2} \int_0^1 \frac {|\kappa (x+sw) - \kappa (x)|}{|w|} |V(x)| |\phi(\gamma(x+w)) - \phi(\gamma(x))| ds dw dx
  \\ & + C  \int_{\mathbb R / l \mathbb Z} \int_{x+w \in B_{\frac \Lambda 2}} \int_0^1 \frac {|\kappa (x+sw) - \kappa (x)|}{|w|^2} |V(x)| |\phi(\gamma(x+w)) - \phi(\gamma(x))| ds dw dx
  \\ & \leq C  \int_{B_{\Lambda/2} (0)} \int_{-\frac l2 }^{\frac l2} \int_0^1 \frac {|\kappa (x+sw) - \kappa (x)|}{|w|} |V(x)| |\phi(\gamma(x+w)) - \phi(\gamma(x))| ds dw dx
  \\ & + C  \int_{\mathbb R / l \mathbb Z} \int_{x+w \in B_{\frac \Lambda 2}} \int_0^1 \frac {|\kappa (x+sw) - \kappa (x+w)|}{|w|^2} |V(x)| |\phi(\gamma(x+w)) - \phi(\gamma(x))| ds dw dx
  \\  & + C  \int_{\mathbb R / l \mathbb Z} \int_{x+w \in B_{\frac \Lambda 2}} \int_0^1 \frac {|\kappa (x+w) - \kappa (x)|}{|w|^2} |V(x)||\phi(\gamma(x+w)) - \phi(\gamma(x))| ds dw dx
  \\  & \leq C  \int_{B_{\Lambda/2} (0)} \int_{-\frac l2 }^{\frac l2} \int_0^1 \frac {|\kappa (x+sw) - \kappa (x)|}{|w|} |V(x)| |\phi(\gamma(x+w)) - \phi(\gamma(x))|ds dw dx
  \\ & + C  \int_{B_{\Lambda / 2 }} \int_{\frac l2} ^{\frac l2} \int_0^1 \frac {|\kappa (x+sw) - \kappa (x)|}{|w|^2} |V(x+w)| |\phi(\gamma(x+w)) - \phi(\gamma(x))|ds dw dx
  \\  & + C  \int_{B_{\Lambda /2 }} \int_{-\frac l2 }^{\frac l2 } \frac {|\kappa (x+w) - \kappa (x)|}{|w|^2} |V(x+w)| |\phi(\gamma(x+w)) - \phi(\gamma(x))| dw dx
  \\ & \leq C \sup_{s_1,s_2 \in [0,1]} \int_{B_{\Lambda /2 }} \int_{-\frac l2 }^{\frac l2 } \frac {|\kappa (x+s_1w) - \kappa (x)|}{|w|^2} |V(x+s_2w)| |\phi(\gamma(x+w)) - \phi(\gamma(x))|  dw dx
 \end{align*}
To estimate this last supremum we decompose the integral into
 \begin{align*}
  & \int_{B_{\Lambda /2 }} \int_{-\frac l2 }^{\frac l2 } \frac {|\kappa (x+s_1w) - \kappa (x)|}{|w|^2} |V(x+s_2w)| |\phi(\gamma(x+w)) - \phi(\gamma(x))|  dw dx
  \\ & \leq \int_{B_{\Lambda /2 }} \int_{-\frac \Lambda 2 }^{\frac \Lambda 2 } \frac {|\kappa (x+s_1w) - \kappa (x)|}{|w|} |V(x+s_2w)|   dw dx \\ & \quad  + \int_{B_{\Lambda /2 }} \int_{|w| \geq \frac  \Lambda 2 } \frac {|\kappa (x+s_1w) - \kappa (x)|}{|w|^2} |V(x+s_2w)|   dw dx
 \end{align*}
Then we can estimate the first term  by
\begin{align*}
 & C\|\kappa\|_{B^{\frac 12 }_{2,2}(B_\Lambda)} \|V\|_{L^2(B_{\Lambda/2})} \leq \varepsilon \tilde S_3
 + C_\varepsilon \|\kappa\|^2_{B^{\frac 1 2}_{2,2}(B_\Lambda)}
 \leq \varepsilon (\tilde S_3 + S_3) + C_\varepsilon 
\end{align*}
where we used the interpolation estimates in Lemme~\ref{lem:Interpolation} and $M_{\frac 32} \leq 1.$
We estimate the second term using Lemma~\ref{lem:EasyEstimate} and then again the interpolation estimates yield
\begin{align*}
  &\quad C \int_{x \in B_{\Lambda /2}} \int_{w \geq \Lambda/ 2}\frac{|\kappa(x+s_1 w)|^2 + |V(x+s_2 w)|^2}{|w|^2} ds \\ & \leq \varepsilon S_3 +  C_\varepsilon \left(  \|\kappa\|^2_{L^2(B_\Lambda(0))} + \sum_{j=1}^\infty \frac { \|\kappa\|^2_{L^2(B_{\Lambda +j} - B_{\Lambda + j -1})} }{(\Lambda + j)^2} \right)
  \\ &\leq \varepsilon S_3 + \varepsilon \|\partial_s^3 \gamma\|^2_{B_\Lambda(0)} + \varepsilon \sum_{j=1}^\infty \frac { \|\partial_s^3 \gamma\|^2_{L^2(B_{\Lambda +j} - B_{\Lambda + j -1})} }{(\Lambda + j)^2} + C_\varepsilon \left(1 + \sum_{j=1}^\infty \frac 1 {(\Lambda + j)^2} \right) \\
  & \leq \varepsilon (S_3 + \tilde S_3) + C_\varepsilon.
\end{align*}
 Hence, 
\begin{align*}
I_{21} \leq \varepsilon (S_3 + \tilde S_3) + C_\varepsilon.
\end{align*}
 
Similarly, we get
 \begin{align*}
  I_{22} &\leq C \int_{\mathbb R / l\mathbb Z} \int_{-\frac l2}^{\frac l2} \int_ {[0,1]^2} \frac {|\gamma'(x+s_1 w) - \gamma'(x+s_2 w)|^2}{|w|^2} |\kappa (x)| | V(x)|  \left| \phi(\gamma(x+w)) - \phi(\gamma(x)|\right) dw dx 
  \\
  &\leq C \int_{\mathbb R / l\mathbb Z} \int_{-\frac l2}^{\frac l2} \int_ {0}^1 \frac {|\gamma'(x+s_1 w) - \gamma'(x)|^2}{|w|^2} | \kappa (x) || V(x)| \left| \phi(\gamma(x+w)) - \phi(\gamma(x))\right| dw dx 
  \\ 
  &\leq \sup_{s_1,s_2 \in [0,1]} C \int_{B_ {\Lambda /2}(0)} \int_{|w| \leq \frac l 2} \int_ {0}^1 \frac {|\gamma'(x+s_1 w) - \gamma'(x)|^2}{|w|^2} |\kappa (x+s_2 w)| \\ & \quad \quad \quad \quad \quad \quad \quad \quad \quad \quad \quad \quad \quad \quad \quad | V(x+s_2 w)| |\phi(\gamma(x+w))-\phi(\gamma(x))|  dw dx
 \end{align*}
 which as above can be estimated by 
\begin{equation*}
  \varepsilon (S_3 + \tilde S_3) + C_\varepsilon.
\end{equation*}
\end{proof}

\begin{lemma} [Estimate for $I_3$] \label{lem:I3} Let $M_{\frac 32 } \leq 1$ and $\gamma(0) \in B_2(0).$
 Given $\varepsilon >0$ we have 
 \begin{equation*}
   |I_3| \leq  \varepsilon  (S_3 + \tilde S_3) + C_\varepsilon.
 \end{equation*}
 for some $C_\varepsilon < \infty$
\end{lemma}

\begin{proof}
 We use that 
 \begin{equation*}
  \left|\phi(\gamma(x+w)) + \phi(\gamma(x)) - 2 \int_0 ^1 \phi(\gamma(x+\tau w)) d \tau \right|
 = C |w|^2
 \end{equation*}
 and for $x\notin B_{\Lambda /2}(0)$
 \begin{multline*}
  \Bigg|\phi(\gamma(x+w)) + \phi(\gamma(x))  - 2 \int_0 ^1 \phi(\gamma(x+\tau w)) d \tau \Bigg| \\ \leq
  \begin{cases}
  0 \text{ for } |w| \leq |x|-2\dist \\
  2  \text{ for } |x|-2\beta \leq |w| \leq |x|+2\dist \\
  \frac 2 {|w|} \text{ for } |x|+2\dist \leq |w|\\
  \end{cases}
  \end{multline*}
 to get
 \begin{align*}
  \int_{\frac l2}^{\frac l2}\frac{\left|\phi(\gamma(x+w)) + \phi(\gamma(x)) - 2 \int_0 ^1 \phi(\gamma(x+\tau w)) d \tau \right|}{w^2} dw \leq \frac C {x^2}
 \end{align*}
 if $|x| \geq \frac \Lambda 2$ and
 \begin{equation*}
  \int_{\frac l2}^{\frac l2}\frac{\left|\phi(\gamma(x+w)) + \phi(\gamma(x)) - 2 \int_0 ^1 \phi(\gamma(x+\tau w)) d \tau \right|}{w^2} dw \leq C
 \end{equation*}
 if $|x| \leq \frac \Lambda 2$.
 These estimates then imply
 \begin{align*}
  I_3 \leq C \left(  \int_{B_{\Lambda /2}(0)} |\kappa(x)|\, |V(x)| dx 
  + \int_{\mathbb R / l\mathbb Z - B_{\Lambda /2}(0)} \frac {|\kappa (x)| \, |V(x)|}{|x|^2} dx\right)
 \end{align*}
 From here again H\"older's inequality together with Lemma~\ref{lem:EasyEstimate} and the interpolation inequalities of Lemma~\ref{lem:Interpolation} imply the assertion of the lemma as in the proof of Lemma~\ref{lem:I2}
\end{proof}

\begin{lemma} [Estimate for $I_4$] \label{lem:I4}
Let $M_{\frac 32 } \leq 1$ and $\gamma(0) \in B_2(0).$ For all $\varepsilon >0$
 \begin{equation*}
   |I_4| \leq  \varepsilon  (S_3 + \tilde S_3) + C_\varepsilon.
 \end{equation*}
 for some $C_\varepsilon < \infty$.
\end{lemma}

\begin{proof}
To estimate $I_4$ we use \eqref{eq:EstEnergyIntegrand} to get
\begin{align*}
|I_4| & = \left|\int_{\mathbb R / l \mathbb Z} \int_{-\frac l 2}^{\frac l2 }   \left( \frac 1 {| \gamma(x+w) -  \gamma(x)|^2} - \frac 1 {|w|^2} \right)\langle V(x), \nabla \phi(\gamma(x))\rangle dwdx \right|
\\ &\leq  C \int_{\mathbb R / l \mathbb Z} \int_{-\frac l 2}^{\frac l2 }  \iint_{[0,1]^2} \frac {|\gamma'(x+s_1 w) - \gamma'(x+ s_2 w)|^2}{|w|^2}
|V(x)| |\nabla \phi(\gamma(x))| ds_1 ds_2 dw dx \\
\\ &\leq  C \int_{B_{2\dist}(0)} \int_{-\frac l 2}^{\frac l2 }  \int_0^1 \frac {|\gamma'(x+s w) - \gamma'(x)|^2}{|w|^2}
|V(x)| ds dw dx  \\ 
  &\leq C   \int_{-\frac l 2}^{\frac l 2} \int_0^1  \left(\int_{x \in B_{2\dist} (0)}  \frac {|\gamma'(x+s_1 w) - \gamma'(x)|^4}{|w|^2} dx \right)^{\frac 1 2} dwds \|V\|_{L^2(B_2(0))}.
\end{align*}
Since
\begin{align*}
  \int_{-\frac l 2}^{\frac l 2} \int_0^1 & \left(\int_{x \in B_{2\dist} (0)}  \frac {|\gamma'(x+s_1 w) - \gamma'(x)|^4}{|w|^2} dx \right)^{\frac 1 2} dwds  \\ & \leq \int_0^1 \int_{-s\frac l 2}^{s\frac l 2}  \left(\int_{x \in B_{2\dist} (0)}  \frac {|\gamma'(x+ w) - \gamma'(x)|^4}{|w|^2} dx \right)^{\frac 1 2} dwds \\ &\leq \int_0^1 \int_{-\frac \Lambda 2}^{\frac \Lambda 2}  \left(\int_{x \in B_2 (0)}  \frac {|\gamma'(x+ w) - \gamma'(x)|^4}{|w|^2} dx \right)^{\frac 1 2} dwds  +  C \\
  & \leq C ( \|\gamma'\|^2_{B^{\frac 1 2}_{4,2} (B_{\Lambda}(0))} + 1 ),
\end{align*}
again the interpolation Lemma~\ref{lem:Interpolation} gives the assertion.
\end{proof}

The final ingredient shows that the summands in $\tilde S_3$ and $S_3$ are essentially the same.

\begin{lemma} \label{lem:STildeSComparable}
Let $M_{\frac 32 } \leq 1$ and $\gamma(0) \in B_2(0).$
For all $\varepsilon >0$ we have 
 \begin{equation*}
  \int_{B_1(0)} |\mathcal H|^2 dx  \leq C \int_{B_{4\beta}(0)} |\partial_s^3 \gamma|^2 ds + \left(C_\varepsilon M^\alpha_{\frac 3 2} + \varepsilon \right) S_3+ C_\varepsilon
 \end{equation*}
 and hence especially
 \begin{equation*}
  \tilde S_3 \leq C( S_3 +1).
 \end{equation*}
 Furthermore, we have
 for all $\varepsilon >0$
 \begin{equation*}
  \int_{B_1(0)} |\partial_s^3 \gamma|^2 dx  \leq C \int_{B_{4\beta}(0)} |\mathcal H \gamma|^2 ds + \left(C_\varepsilon M^\alpha_{\frac 3 2} + \varepsilon \right)  S_3+ C_\varepsilon
 \end{equation*}
 and hence especially
 \begin{equation*}
   S_3 \leq C \tilde  S_3 + \left(C_\varepsilon M^\alpha_{\frac 3 2} + \varepsilon \right) S_3+ C_\varepsilon.
 \end{equation*}
 for some $C_\varepsilon<\infty$ depending $\varepsilon$ and the bi-Lipschitz constant of $\gamma$. If $M_{\frac 32}$ is small enough, we have
 \begin{equation*}
  S_3 \leq C (\tilde S_3 +1).
 \end{equation*}

\end{lemma}

\begin{proof}
 Lemma~\ref{lem:PrincipalPart} tells us that 
  \begin{equation*}
   \Bigg|\int_{\mathbb R / \mathbb Z} |\mathcal H \gamma|^2 \phi (\gamma) dx -  \int_{\mathbb R / \mathbb Z} |Q \gamma|^2 \phi (\gamma) dx \Bigg| 
   \leq  C\left(M^\alpha_{\frac 3 2} + \varepsilon \right) S_3(x)+ C_\varepsilon,
  \end{equation*}
  and hence especially
  \begin{equation} \label{eq:Comp1}
    \Bigg|\int_{B_1(x)} |\mathcal H \gamma|^2  dx -  \int_{B_1(0)} |Q \gamma|^2 dx \Bigg| 
   \leq  C\left(M^\alpha_{\frac 3 2} + \varepsilon \right) S_3(x)+ C_\varepsilon.
  \end{equation}

  Let $\psi \in C^\infty(\mathbb R)$ be such that $\chi_{B_{2}(0)} \leq \phi \leq \chi_{B_{4}(0)}$. We get
  \begin{multline*}
   \|\tilde Q(\kappa)\|_{L^2(B_{2}(0))}  = \| \psi \tilde Q(\kappa)\|_{L^2(B_{2}(0))} \\ \leq \|\tilde Q[\psi \kappa] - \psi \tilde Q[\kappa] - \kappa \tilde Q[\psi]\|_{L^2(B_{2}(0))}  + \|\tilde Q[\psi \kappa]\|_{L^2(B_{2}(0))} \\ + \|\kappa \tilde Q [\psi]\|_{L^2(B_{2}(0))}.
  \end{multline*}
  The commutator estimate (Lemma~\ref{lem:Commutator}) and the interpolation estimate (Lemma~\ref{lem:Interpolation}) tell us that
  \begin{align*}
   & \|\tilde Q[\psi \kappa] - \psi \tilde Q[\kappa] - \kappa Q[\psi]\|^2_{L^2(B_{2}(0))} \\&\leq C \|\kappa\|^2_{B^{\frac 1 2}_{4,2}(B_{\Lambda}(0))} \|\psi\|^2_{B^{\frac 1 2}_{4,2}B_{\Lambda}(0))} +C \sum_{j=1}^\infty \frac {\|\kappa\|^4_{L^4(B_{\Lambda + j}\setminus B_{\Lambda +j-1})}} {(\Lambda +j)^2}
   \\ &\leq \varepsilon S_3 + C_\varepsilon.
  \end{align*}
  As by Lemma~\ref{lem:FourierCoefficientsQ}
  \begin{align*}
   \|\tilde Q[\psi \kappa]\|_{L^2(B_{2}(0))} & \leq C \|\psi \kappa\|_{W^{1,2}(\mathbb R / l \mathbb Z)} \leq C \left( \|\partial_s \psi\kappa\|_{L^2(\mathbb R / l \mathbb Z} + \|\psi \kappa\|_{L^2}  \right) \\ &\leq C \left( \|\partial^3 \gamma\|_{L^2(B_{4 }(0)} + \|\kappa\|_{L^2 (B_2(0))} \right) 
  \end{align*}
  and
  \begin{equation*}
   \|\kappa \tilde Q [\psi]\|_{L^2(B_{2 }(0))} \leq C \|\kappa\|_{L^2(B_{2 }(0))},
  \end{equation*}
  we get using again the interpolation estimates
  \begin{equation}  \label{eq:Comp2}
   \| Q(\gamma)\|^2_{L^2(B_{2}(0))} =\|\tilde Q(\kappa)\|^2_{L^2(B_{2}(0))} \leq  C \|\partial_s^k  \gamma\|_{L^2(B_{4}(0)} + \varepsilon S_3 + C_\varepsilon.
  \end{equation}
  The estimates \eqref{eq:Comp1} and \eqref{eq:Comp2} imply the first inequality. Summing up yields the second.
  
  On the other hand,  for a cutoff function $\psi \in C^\infty(\mathbb R)$ such that $\chi_{B_{\frac 1 2 }(0)} \leq \psi \leq \chi_{B_{1}(0)}$ we have 
  \begin{equation*}
   \|Q\gamma\|_{L^2 (B_{1} (0))} \geq \|\psi Q \gamma\|_{L^2 (B_{1}(0))}
  \end{equation*}
  which implies as above
  \begin{equation*}
   \|Q \gamma\|_{L^2 (B_{1}(0))} \geq \|Q(\psi \kappa )\|_{L^2} - \varepsilon S_3 + C_\varepsilon.
  \end{equation*}
  Using Lemma~\ref{lem:FourierCoefficientsQ} we get
  \begin{equation*}
   \|\nabla(\psi \kappa)\|^2_{L^2} \leq \|Q(\psi \kappa )\|^2_{L^2} + \|\kappa\|^2_{L^2(B_1(0))} + \varepsilon S_3 + C_\varepsilon
  \end{equation*}
  and hence using an interpolation estimate
  \begin{align*}
   \|\nabla \kappa\|^2 _{L^2 (B_{\frac 1 2}(0))} \leq \|Q(\psi \kappa )\|^2_{L^2(B_1(0))} + C \|\kappa\|^2_{L^2(B_1(0))} +\varepsilon S_3 + C_\varepsilon \\
    \leq C \|Q(\psi \kappa )\|^2_{L^2(B_1(0))} + \varepsilon S_3 + C_\varepsilon.
  \end{align*}
  Using \eqref{eq:Comp1} we obtain
  \begin{equation*}
   \|\nabla \kappa\|^2_{L^2 (B_{\frac 1 2} (0))} \leq \|\mathcal H \gamma \|^2_{L^2 (B_1(0))} + C\left(M^\alpha_{\frac 3 2} + \varepsilon \right) S_3(x) + C_\varepsilon,
  \end{equation*}
  and covering the ball $B_1(0)$ by balls of radius $\frac 12$ we get
  \begin{equation*}
   \|\nabla \kappa\|^2_{L^2 (B_{1} (0))} \leq \|\mathcal H \gamma \|^2_{L^2 (B_2(0))} + C\left(M^\alpha_{\frac 3 2} + \varepsilon \right) S_3(x) + C_\varepsilon
  \end{equation*}
  This implies the remaining three inequalities of the lemma.
 \end{proof}

Gathering all the estimates above, we can now show
 \begin{lemma} [Differential inequality] \label{lem:APrioriEstimateI}
  For $1>\varepsilon >0$ there is a constant $C_\varepsilon < \infty$ such that
  \begin{equation*}
   \frac{d}{dt} E_\phi (\gamma_t)  +  \int_{\mathbb R / l \mathbb Z}  |\mathcal H \gamma_t|^2 \phi \leq
   \left(C_\varepsilon M^\alpha_{\frac 32}(t) + \varepsilon \right) \tilde S^{ext}_3(0,t) + C_\varepsilon
  \end{equation*}
  whenever $M_{\frac 32}$ is sufficiently small.
 \end{lemma}
\begin{proof}
 If $\gamma_t(\mathbb R / \mathbb Z) \cap B_2(0) = \emptyset$ we have 
 \begin{equation*}
   \frac{d}{dt} E_\phi (\gamma_t)  =- \int_{\mathbb R / l \mathbb Z}  |\mathcal H \gamma_t|^2 \phi,
  \end{equation*}
since both sides of the equation are vanishing.
Let us now assume that $\gamma(x_t,t) \in B_2(0)$  for some $x_t \in \mathbb R / \mathbb Z$. Then the Lemmata~\ref{lem:PrincipalPart}, \ref{lem:I2}, \ref{lem:I3}, \ref{lem:I4}, and \ref{lem:STildeSComparable} tell us that
 \begin{equation*}
   \frac{d}{dt} E_\phi (\gamma_t)  +  \int_{\mathbb R / l \mathbb Z}  |\mathcal H \gamma_t|^2 \phi \leq
   \left(C_\varepsilon M^\alpha_{\frac 32}(t) + \varepsilon \right) \tilde S_3(x_t,t) + C_\varepsilon.
  \end{equation*}
 It is an easy exercise to show using the bi-Lipschitz estimate that
  \begin{equation*}
  \tilde S_3(x_t,t) \leq C \tilde S_3^{ext}(0,t)
  \end{equation*}
  where the constant $C$ depends on the bi-Lipschitz constant of $\gamma$.
  Hence,
  \begin{equation*}
  \frac{d}{dt} E_\phi (\gamma_t)  + \int_{\mathbb R / l \mathbb Z}  |\mathcal H \gamma_t|^2 \phi \leq
   \left( C_\varepsilon M^\alpha_{\frac 32}(t) + \varepsilon \right) \tilde S^{ext}_3(0,t) + C_\varepsilon
  \end{equation*}
\end{proof}

 Exploiting this result, we get

\begin{proposition} \label{prop:estimateLocalDensity}
For every $\delta >0$ there are constants $ \varepsilon_0>0 $ and $C <\infty $ such that
$
 \sup_{x \in \mathbb R ^n} E_{B_1(x)} (\gamma_{t_0})\leq \varepsilon_0
$
for some $t_0 \in [0,T)$ implies 
\begin{equation*}
        \int_{t_0}^t \int_{\gamma_t^{-1}(B_{1}(x))}|\mathcal H \gamma_\tau|^2dsd\tau \leq C
 \quad \quad \text{and} \quad \quad 
E_{B_1 (x) }(\gamma_\tau ) \leq  \delta
\end{equation*}
for all $\tau \in [t_0,\min \{T,t_0+1\})$ and $x\in \mathbb R^n$.
\end{proposition}

\begin{proof}
  We assume without loss of generality that $t_0=0$. Clearly we only have to show the claim under the additional assumption that $\delta>0$ is small. Furthermore, it is enough to show that \begin{equation*}
        \int_{t_0}^t \int_{\gamma_t^{-1}(B_{1}(x))}|\partial_s ^3 \gamma_\tau|^2dsd\tau \leq C
 \quad \quad \text{and} \quad \quad 
E_{B_1 (x) }(\gamma_\tau ) \leq  \delta
\end{equation*}
for all $\tau \in [t_0,\min \{T,t_0+\varepsilon_2\})$ and $x\in \mathbb R^n$ for a sufficiently small $\varepsilon_2$. 
One then obtains the assertion in its original form by applying the preliminary result to the rescaled flow
$$
 \tilde \gamma (x,t) := \frac 1 {\sqrt[3]{\varepsilon_2}}\gamma(\frac x {\sqrt[3]{\varepsilon_2}}, \frac t {\varepsilon_2})
$$
that satisfies by a standard covering argument
$$
 E_{B_1(0)}(\tilde \gamma) \leq C_n \frac 1 {(\varepsilon_2)^{\frac n3}} \varepsilon_1.
$$

  Lemma~\ref{lem:APrioriEstimateI} tells us that
  \begin{equation} \label{eq:EnergyEstimate}
   \frac{d}{d t} E_\phi (\gamma_t)  +  \int_{\mathbb R / l \mathbb Z} \int_{B_{1/4}(0)} |\mathcal H \gamma|^2 \phi \leq
   \left( C_\varepsilon M^\alpha_{\frac 3 2} + \varepsilon \right) S^{ext}_3+ C_\varepsilon.
  \end{equation}
  
  Let us assume that $t_1$ is the first time such that
  \begin{equation*}
   \sup_{x \in \mathbb R^n} E_{B_1}(\gamma_{t_1}) \geq \delta.
  \end{equation*}
  We set
  \begin{equation*}
   IM_3:= \sup_{x\in \mathbb R^n} \int_{0}^{t_1} \int_{\gamma^{-1}(B_{1/4}(x))} |\mathcal H  \gamma|^2
   ds d\tau.
  \end{equation*}
  After a translation we can assume that
  \begin{equation*}
   IM_3 = \int_{0}^{t_1} \int_{\gamma^{-1}(B_{1/4}(0))} |\mathcal H \gamma|^2
   ds d\tau 
  \end{equation*}
  Due to the definition of $\tilde S_3(x)$ and Lemma~\ref{lem:STildeSComparable} we know that
  \begin{equation*}
   \int_{t_0}^{t_1} \tilde S_3 d\tau \leq C \int_{0}^{t_1}  \tilde S_3 d \tau +
  C(M_{\frac 32}^\alpha + \varepsilon)  \int_{0}^{t_1}  \tilde S_3 d \tau + C  t_1
  \leq C IM_3 + C (t_0- t_1).
  \end{equation*}

  Integrating \eqref{eq:EnergyEstimate} and using $\chi_{B_1(0)} \leq \phi \leq \chi_{B_2(0)}$  we hence get
  \begin{equation} \label{eq:IntegratedEst}
   E_\phi(\gamma_{t_1}) + c_0 IM_3 \leq E_\phi(\gamma_{0}) + C (\delta^\alpha + \varepsilon) IM_3 + C_\varepsilon t_1 \leq \varepsilon_0 + C (\delta^\alpha + \varepsilon)  IM_3 + C_\varepsilon t_1
  \end{equation}
  If $C(\delta ^\alpha + \varepsilon) \leq  \frac {c_0}2$,  this implies
  \begin{equation*}
   \frac {c_0}2 IM_3 \leq \varepsilon_0  + C t.
  \end{equation*}
  Plugging this back into the inequality \eqref{eq:IntegratedEst}, we get for all $x \in \mathbb R^n$
  \begin{equation*}
   E_{B_1(x)}(\gamma_{t_1}) \leq E_{\phi_x}(\gamma_{t_1}) \leq \varepsilon_0 + C(\delta^\alpha + \varepsilon) ( \varepsilon_0 + Ct) + C_{\varepsilon} t_1  < \delta
  \end{equation*}
  if we first chose $\varepsilon_0>0$ and then $t$  small enough.
\end{proof}

\subsection{Estimating the elastic energy}
In this section we derive estimates from the evolution equations of energies containing
higher order terms. The following lemma was proven in \cite{Blatt2016}:

\begin{lemma} [Evolution of Higher order energies] \label{lem:EvolutionOfHigherOrderEnergies}
 Let $\gamma$ be a family of curves moving with normal speed
 $V$. Then
\begin{multline} \label{eq:EvolutionOfHigherOrderEnergies}
  \partial_t \int_{\mathbb R / \mathbb Z} |\partial_s^{k} \kappa|^2  \phi ds
  = 2 \int_{\mathbb R / \mathbb Z}
  \langle \partial_s^{k+2}V,\partial_s^k \kappa\rangle \phi ds 
 + 2\int
  \langle P^{k}_2(V,\kappa)  \tau, \partial^{k+1}_s
\kappa\rangle \phi ds 
\\
 \quad + 2 \int \langle P_3^{k}(V,\kappa), \partial_s^{k} \kappa
\rangle \phi ds
- \int |\partial_s^k \kappa|^2 \langle \kappa ,  V\rangle \phi ds
+ \int_{\mathbb R / l \mathbb Z} |\partial^k_s \kappa| \nabla_{V} \phi ds.
\end{multline}
\end{lemma}

In this section, we will derive estimates for the right-hand side of equation \eqref{eq:EvolutionOfHigherOrderEnergies} for the case that $V= - \mathcal H$.
We use both the evolution equations from Lemma~\ref{lem:EvolutionOfHigherOrderEnergies} and these estimates to bound the so-called \emph{elastic energy} of the curve $\gamma$, i.e. the $L^2$-norm of its curvature.

\begin{proposition} [Estimate for the elastic energy] \label{prop:estimateWillmore}
 Let $\gamma: [0,T) \times \mathbb R / \mathbb R \rightarrow \mathbb R^n$, $T>1$ be a smooth solution  of \eqref{eq:GFME}.
 There is an $\varepsilon_0 >0$ depending only on $n$ such that 
 $$
   \sup_{(x,t)\in \mathbb R^n \times (0,1)} E_{B_1(x)}(\gamma(\cdot,t)) < \varepsilon_0
 $$
 implies 
 $$
  \sup_{x \in \mathbb R / \mathbb Z} \int_{B_1(x) \cap \gamma_1} |\kappa_{\gamma_1}|^2 ds \leq C
 $$
 and
 $$
  \inf_{t \in [0, 1]} \int_{B_1(x)} |\partial_s \kappa_{t}|^2 ds \leq C.
 $$
\end{proposition}

\subsubsection {Preliminary estimates}

To estimate the respective integrals appearing on the right-hand side of equation \eqref{eq:EvolutionOfHigherOrderEnergies} we have to distinguish as before between $|w|$ big and $|w|$ small. The next lemma helps us to deal with the part where $|w|$ is big:

\begin{lemma} \label{lem:SmallW}
 Let us assume that $p \in [1,\infty)$, $l_i \in \mathbb N$, $l_i \geq 2$, $p_i \in [1, \infty)$ for  $i=1,\ldots r$,
 and let $\in \mathbb N$ be chosen such that
 $$
   l_i \leq m,\quad l_i - \frac 1 {p_i} \leq  m -\frac 1 2,  \quad \text{ and } \quad \sum_{i=1}^n \frac 1 {p_i} = \frac 1 p.
 $$
For $\Lambda = 1000 \cdot 18 e^{\frac{E(\gamma_0)}4 }$ we set
  \begin{equation*}
  g(x) := \int_{l \geq |w| \geq \Lambda} \int_{s \in [0,1]^r}\frac {\prod_{i=1}^r \left| \partial^{l_i} \gamma(x+ s_i w ) \right| }{|w|^2} ds dw.
 \end{equation*}
 and assume that $M_{\frac 32 } \leq 1$.
Then there is a constant $\beta_1, \beta_2 >0$ such that
 \begin{equation*}
  \|g\|^p_{L^p(B_1(0))} \leq C  (M_{3/2}^{\beta_1} + M_{3/2}^{\beta_2}) \sum_{i=1}^r \left(  \|\partial^m \gamma\|^{\theta_i}_{L^2 (B_\Lambda (0))}
   + \sum_{j=1}^\infty \frac {\|\partial^{m} \gamma\|^{\theta_i}_{L^{2}(B_{\Lambda + j}(0)\setminus B_{\Lambda + j -1 }(0))}}{(\Lambda + j)^2} \right) + C.
 \end{equation*}
where $\theta_i = p \frac {l_i -\frac 1 {p_i}}{m-\frac 1 2}$ and $C< \infty$ only depends on $n$ and $E(\gamma_0).$
\end{lemma}

\begin{proof}
Using Jensen's inequality, we obtain
\begin{align*}
 \int_{B_1(0)} |g(x)|^p dx  &=   \int_{B_1(0)} \left( \int_{l \geq |w| \geq \Lambda}  \int_{s \in [0,1]^r}\frac {\prod_{i=1}^r \left| \partial^{l_i} \gamma(x+ s_i w ) \right| }{|w|^2} ds dw \right)^p dx \\
 & \leq C \int_{B_1(0)} \int_{l \geq |w| \geq \Lambda}  \int_{s \in [0,1]^r}\frac {\prod_{i=1}^r \left| \partial^{l_i} \gamma(x+ s_i w ) \right| ^p }{|w|^2} ds dw  dx .
\end{align*}
As $\sum_{i=0}^r \frac 1 {p_i} =\frac 1 p$, we get by Cauchy's inequality
\begin{equation} \label{eq:Estimate1}
 \int_{B_1(0)} |g(x)|^p dx 
 \leq \frac C {\Lambda ^{p-1}} \int_{B_1(0)} \int_{l \geq |w| \geq \Lambda}  \int_{s \in [0,1]^r} \frac {\sum_{i=1}^r \left| \partial^{l_i} \gamma(x+ s_i w ) \right| ^{p_i} }{|w|^2} ds dw  dx .
\end{equation}
We can estimate the summands further substituting $\tilde w = s w$ by
\begin{equation} \label{eq:Estimate2}
\begin{aligned}
 \int_{B_1(0)} \int_{l \geq |w| \geq \Lambda}  & \int_0^1 \frac {\left| \partial^{l_i} \gamma(x+ s_i w ) \right| ^{p_i} }{|w|^2} ds d w dx
  \\ & \leq  \int_{B_1(0)} \int_0^1 s \int_{\Lambda \geq |\tilde w| \geq s \Lambda}    \frac {\left| \partial^{l_i} \gamma(x+ \tilde w ) \right| ^{p_i} }{|\tilde w|^2} ds d\tilde w dx \\ & \quad \quad \quad \quad \quad + \int_{B_1(0)}\int_0^1 \int_{l \geq |\tilde w| \geq \Lambda}    \frac {\left| \partial^{l_i} \gamma(x+ \tilde w ) \right| ^{p_i} }{|\tilde w|^2} ds d \tilde w dx
  \\
 & \leq  C \|\partial^{l_i} \gamma\|^{p_i }_{L^{p_i}(B_{\Lambda}(0))} +  \int_{l \geq |y| \geq \Lambda/2}    \frac {\left| \partial^{l_i} \gamma(y) \right| ^{p_i} }{|y|^2} dy.
\end{aligned}
\end{equation}
Applying the Gagliardo-Nirenberg type inequality (Lemma~\ref{lem:Interpolation}), we obtain for
$$
  \theta_i = p_i  \frac {l_i -1  - \frac 1 {p_i}}{m-\frac 1 2 }
$$
that
\begin{align*}
 \|\partial^{l_i} \gamma\|^{p_i}_{L^{p_i}(B_1(0))} & \leq C \|\partial^m \gamma\|^{\theta_i}_{L^2(B_1(0))} M_{3/2}^{\frac {p_i  - \theta_i} 2} + M_{3/2}^{\frac {p_i}2}.
 \end{align*} 
Scaling this inequality, we get
\begin{align*}
\|\partial^{l_i} \gamma\|^{p_i}_{L^{p_i }(B_\Lambda(0))} & \leq C \left(\|\partial^m \gamma\|^{\theta_i}_{L^2(B_\Lambda(0))}  M_{3/2}^{\frac {p_i  - \theta}2} + M_{3/2}^{\frac{p_i}2 } \right).
\end{align*}
Furthermore,
\begin{multline*}
 \int_{l \geq |y| \geq \Lambda/2}    \frac {\left| \partial^{l_i} \gamma(y) \right| ^{p_i} }{|y|^2} dy
 \leq \frac C {\Lambda^2} \|\partial^{l_i} \gamma\|^{p_i}_{L^{p_i }(B_{\Lambda}(0))} + C \sum_{j=1}^\infty \frac {\|\partial^{l_i} \gamma\|^{p_i}_{L^{p_i}(A_{\Lambda + j}(0))}}{(\Lambda + j)^2}
 \\  \leq  C  M_{3/2}^{\frac{p_i  - \theta}2 } \left(  \|\partial^{l_i} \gamma\|^{\theta_i}_{L^2(B_\Lambda (0))} +  \sum_{j=1}^\infty  \|\partial^{l_i} \gamma\|^\theta_{L^2(B_\Lambda (0))} \right) \\ +C   M_{3/2} ^{\frac {p_i}2}
\end{multline*}
From \eqref{eq:Estimate1} and \eqref{eq:Estimate2} the assertion follows.
\end{proof}

The second ingredient is the following lemma, which helps to deal with small $|w|$.

\begin{lemma} \label{lem:BigW}
 Let
 \begin{multline*}
  g(x) := \int_{ |w| \leq \Lambda} \int_{s \in [0,1]^r} \int_{\tau_1, \tau_2 \in [0,1]} \prod_{i=1}^r \left| \partial^{l_i} \gamma(x+ s_i w ) \right|  \\  \frac {|\partial^{l_{r+1}} \gamma(x+\tau_1 w )-  \partial^{l_{r+1}}\gamma (x+\tau_2 w)| \, |\partial^{l_{r+2}} \gamma(x+\tau_1 w )-  \partial^{l_{r+2}}\gamma (x+\tau_2 w)|}{w^2} ds d\tau_1 d \tau_2 dw.
 \end{multline*}
 Let $\tilde l_i = l_i$ for $i=1, \ldots r$ and $\tilde l_i=l_i + \frac 1 2$ for 
 $i=r+1,r+2$. If $\tilde l_i \leq m$ and $ \tilde l_i - \frac 1{p_i} < m - \frac 1 2$ for all $i=1,\ldots r+2$, then there is a constant $\alpha >0$ such that
 \begin{equation*}
  \|g\|_{L^p(B_{\Lambda})} \leq (C_\varepsilon M_{3/2}^{\alpha} + \varepsilon) (\|\partial^m \gamma\|^\theta_{L^2 (B_{2\Lambda}} + C_{\varepsilon})
 \end{equation*}
 where
 \begin{equation*}
  \theta = \frac{\sum_{i=1}^{r+2} (\tilde l_i - 1) - \frac 1 p } {m-\frac 3 2}.
 \end{equation*}
\end{lemma}

\begin{proof}
 We write
 \begin{equation*}
  g(x) = \int_{I_l} \int_{s \in [0,1]^r} \iint_{\tau_1, \tau_2 \in [0,1]}g_{x,\tau_1, \tau_2} (x,w)  ds d\tau_1 d \tau_2 dw
 \end{equation*}
 where
 \begin{multline*}
  g_{s,\tau_1, \tau_2} (x,w) :=\prod_{i=1}^r \left| \partial^{l_i} \gamma(x+ s_i w ) \right|  \\  \frac {|\partial^{l_{r+1}} \gamma(x+\tau_1 w )-  \partial^{l_{r+1}}\gamma (x+\tau_2 w)| \, |\partial^{l_{r+2}} \gamma(x+\tau_1 w )-  \partial^{l_{r+2}}\gamma (x+\tau_2 w)|}{w^2}.
 \end{multline*}
 Using H\"older's inequality, we get for $|w| \leq \Lambda$
 \begin{align*}
  \|g_{s,\tau_1, \tau_2}(\cdot,w)\|_{L^p (B_{\Lambda}(0))}
  &\leq \prod_{i=1}^r \|\partial^{l_i} \gamma \|_{L^{p_i}(B_{2\Lambda}(0))}
  \frac{\prod_{j=1,2}\|\partial^{l_{r+j}} \gamma(\cdot +(\tau_1 - \tau_2) w )-\partial^{l_{r+j}} \gamma\|_{L^{p_{r+j}}(\Lambda)} }{w^2}
 \end{align*}
 Since
 \begin{align*}
   & \iint_{\tau_1, \tau_2 \in [0,1]} \int_{|w| \leq \Lambda} \frac{\prod_{j=1,2}
  \|\partial^{l_{r+j}} \gamma(\cdot +(\tau_1 - \tau_2) w )-\partial^{l_{r+j}} \gamma\|^2_{L^{p_{r+j}}(\Lambda)} }{w^2} dw d\tau_1 d\tau_2\\
  & \leq  
  \iint_{\tau_1, \tau_2 \in [0,1]} \prod_{j=1,2} \left(\frac{  \int_{|w| \leq \Lambda} 
  \|\partial^{l_{r+j}} \gamma(\cdot +(\tau_1 - \tau_2) w )-\partial^{l_{r+j}} \gamma\|^2_{L^{p_{r+j}}(\Lambda)} }{w^2} dw \right)^{\frac 1 2} d\tau_1 d\tau_2 \\
  &\leq \iint_{\tau_1, \tau_2 \in [0,1]} |\tau_1- \tau_2|^{\frac 1 2}\prod_{j=1,2} \left(\frac{  \int_{|\tilde w| \leq \Lambda} 
  \|\partial^{l_{r+j}} \gamma(\cdot +\tilde w )-\partial^{l_{r+j}} \gamma\|^2_{L^{p_{r+j}}(\Lambda)} }{\tilde w^2} d\tilde w \right)^{\frac 1 2} d\tau_1 d\tau_2  \\ &\leq \prod_{j=1,2} \|\partial^{l_{r+j}} \gamma\|_{B (B_{3\Lambda}(p)},
 \end{align*}
 we obtain
 \begin{align*}
  \|g\|_{L^p(B_{\Lambda}(0))} \leq \prod_{j=1,2} \prod_{i=1}^r \|\partial^{l_i} \gamma \|_{L^{p_i}(B_{2\Lambda}(0))}  \|\partial^{l_{r+j}} \gamma \|_{B_{\frac 1 2}^{p_{r+j}, 2}(B_{4 \Lambda} (0))}.
 \end{align*}
 
 As above, the assertion now follows from the Gagliardo-Nirenberg interpolation estimates in Lemma~\ref{lem:Interpolation}.
\end{proof}

\subsubsection{Estimating the derivatives of $\mathcal H$}

For $k \in \mathbb N_0$, $s \in (0,1)$ we define
\begin{align*}
 S_{k+s}(x) & =  \iint_{B_\Lambda (x)}  \frac {|\partial^{k}\gamma(y) - \partial^{k}\gamma(z)|^2} {|y-z|^{1+2s}} dz dy \\
 & \quad \quad +  \sum_{j=1}^\infty \frac 1 {(\Lambda + j)^2} \iint_{B_{\Lambda+j}(x) \setminus B_{\Lambda+j-1}(x)}  \frac {|\partial^{k}\gamma(y) - \partial^{k}\gamma(z)|^2} {|y-z|^{1+2s}} dz dy, \\
  S_{k}(x) & = \|\partial_s^k \gamma\|^2_{L^2(B_\Lambda (0)}  +  \sum_{j=1}^\infty \frac { \|\partial_s^k \gamma\|^2_{L^2 (B_{\Lambda+j}(x) \setminus B_{\Lambda+j-1}(x))} }   {(\Lambda + j)^2}, \\
 \intertext{and}
 \tilde M _{k + \frac 32 } = \tilde M^\phi_{k+\frac 32 } (x) &= \int_{\mathbb R / l \mathbb Z} \int_{-1}^1  \frac {|\partial^{k+1}\gamma(y+w) - \partial^{k+1}\gamma(y)|^2} {w^2} \phi_x (\gamma(y)) dw dy.
\end{align*}
As before, we will assume that $\gamma(0) \in B_2(0)$ to get some preliminary estimates in terms of the intrinsically defined quantities above.
In the final differential inequality we will use the extrinsic quantity
\begin{align*}
 S^{ext}_{k+s}(x) & =  \iint_{\gamma^{-1}(B_{1}(x))}  \frac {|\partial^{k}\gamma(y) - \partial^{k}\gamma(z)|^2} {|y-z|^{1+2s}} dz dy \\
 & \quad \quad +  \sum_{j=1}^\infty \frac 1 {j^2} \iint_{\gamma^{-1}(B_{j+1}(x) \setminus {B_{j}(x)})}  \frac {|\partial^{k}\gamma(y) - \partial^{k}\gamma(z)|^2} {|y-z|^{1+2s}} dz dy, \\
\end{align*}
in place of $S_{k+s}(0)$.

We start with an estimate for $\tilde { \mathcal H }$.  Again we use the decomposition
 \begin{equation*}
  \tilde {\mathcal H} \gamma (x) = Q \gamma(x) + R_1 \gamma(x) + R_2 \gamma(x) = Q \gamma (x) + R \gamma(x)
 \end{equation*}
 where
 \begin{align*}
  Q \gamma(x) &= 2 \lim_{\varepsilon \downarrow 0} \int_{I_{l, \varepsilon}} \left(2  \frac {\gamma(x+w) - \gamma(x) - w \gamma'(x)} {w^4} - \frac {\kappa (x)} {|w|^2} \right) dw
  \\ & = 4 \lim_{\varepsilon \downarrow 0} \int_{I_{l, \varepsilon}} \int_0^1 (1-s)\frac{ \kappa(x+sw) - \kappa (x)} {|w|^2}  dw = \tilde Q \kappa  (x),\\
  R_1 \gamma (x) &=  4 \int_{I_l}   (\gamma(x+w) - \gamma(x) - w \gamma'(x)) \left(\frac 1 {|\gamma(x+w) - \gamma(x) |^4} -  \frac 1  {w^4}  \right) dw ,\\
  R_2 \gamma (x) &=  2 \int_{I_l}   \kappa(x) \left( \frac 1  {w^2} - \frac 1 {|\gamma(x+w) - \gamma(x)|^2} \right) dw
 \end{align*}
 and set $R=R_1 + R_2.$

\begin{lemma} \label{lem:EstimateL2TildeH}
 Let $M_{\frac 32 } \leq 1 $ and $\gamma(0)\in B_2(0)$. For all $ \tilde k\geq k$ there is a constant $\alpha >0$ that for all $\varepsilon >0$
 \begin{equation*}
  \|\partial_s ^k P^\bot_{\gamma'}(R) \|_{L^2(B_1(0))} + \|\partial_s ^k R \|^2_{L^2(B_1(0))} \leq (C_\varepsilon M_{\frac 32} ^{\alpha_1} + \varepsilon) S_{\tilde k+2}^\theta(0) + C_\varepsilon
 \end{equation*}
 for some constant $C_\varepsilon < \infty$, where $\theta = \frac {2k+3}{2\tilde k+1} $. Hence, for every $\varepsilon >0$ and $k_1 >k$
 there is an $C_\varepsilon <\infty$ such that
  \begin{equation*}
  \|\partial_s ^k \tilde{\mathcal H} - \partial_s^k Q \|^2_{L^2(B_1(0))} \leq  \varepsilon S_{\tilde k+3}+ C_\varepsilon.
 \end{equation*}
\end{lemma}

\begin{proof}
First we will show that the two summands building $R$ can be brought into a common form 
and can thus be dealt with simultaneously.

To this end we first use Taylor's theorem to rewrite
\begin{equation*}
 R_1 \gamma (x) = 4 \int_{I_l} \int_0^1 \kappa (x+sw) \left(\frac 1 {|\gamma(x+w) - \gamma(x) |^4} -  \frac 1  {w^4}  \right) ds dw.
\end{equation*}

For $\beta>0$ we observe
\begin{multline*}
 \frac 1 {|\gamma(u+w) - \gamma(u)|^\beta} - \frac 1 {|w|^\beta} 
  = \frac {|w|^\beta}{|\gamma(u+w)-\gamma(u)|^\beta} \cdot
     \frac {1- \frac {|\gamma(u+w)-\gamma(u)|^\beta}{|w|^\beta}} {|w|^\beta} \\
  = G^{(\beta)}\left(\frac {\gamma(u+w)-\gamma(u)}{w}\right)
     \frac {2-2\frac {|\gamma(u+w)-\gamma(u)|^2}{w^2}} {|w|^\beta} \\
  = \int_{0}^1 \int_{0}^1 G^{(\beta)}\left(\frac {\gamma(u+w)-\gamma(u)}{w}\right)
  \frac {|\gamma'(u+\tau_1 w) - \gamma'(u+\tau_2 w)|^2}{|w|^\beta} d\tau_1 d \tau_2
\end{multline*}
where
\begin{align*}
 G^{(\beta)}(z) := \frac 1 {2|z|^\beta} \cdot \frac {1- |z|^\beta}{1 - |z|^2}
\end{align*}
is an analytic function away from the origin. Defining
\[ g^{(\alpha,\beta)}_{s_1, \tau_1, \tau_2}(u,w) :=  G^{(\beta)} 
\left(\frac {\gamma(u+w) - \gamma(u)}{w} \right) \frac {|\gamma'(u+\tau_1 w) -\gamma'(u+\tau_2 w)|^2}{|w|^\alpha} 
\kappa(u+s_1 w) \]
we thus get
\begin{multline}\label{eq:DecompositionOfR}
 R \gamma (x) = 4 \int_{w \in I_l} \iint_{[0,1]^2} \int_0^1 g^{4,2}_{s_1, \tau_1, \tau_2} (x,w) d\tau_1 d\tau_2 d s dw  \\- 2 \int_{w \in I_l} \iint_{[0,1]^2} g^{2,2}_{0, \tau_1, \tau_2} (x,w) d\tau_1 d\tau_2 dw.
\end{multline}
We now give the details of the estimate for the first term. The second term can be estimated analogously.

We differentiate under the integral to get
\begin{align*}
 \partial^k R_1 \gamma (x) &= 4 \int_{w \in I_l} \iint_{[0,1]^2} \int_0^1 \partial^k_x g^{4,2}_{s_1, \tau_1, \tau_2} (x,w) d\tau_1 d\tau_2 d s dw 
 \\ &= 4 \int_{|w| \geq \Lambda} \iint_{[0,1]^2} \int_0^1 \partial^k_x g^{4,2}_{s_1, \tau_1, \tau_2} (x,w) d\tau_1 d\tau_2 d s dw
 \\ & \quad \quad +\int_{|w| \leq \Lambda} \iint_{[0,1]^2} \int_0^1 \partial^k_x g^{4,2}_{s_1, \tau_1, \tau_2} (x,w) d\tau_1 d\tau_2 d s dw
\end{align*}
The product rule and Fa\`a di Bruno's formula tell us that 
\begin{align*}
 \partial_x^k g^{2,2}_{s_1,\tau_1, \tau_2} (x,w) = 
 \sum_{l_1+l_2+l_3+l_4=k} \left( \sum_{\pi \in \Pi_{l_1}} (\partial^{|\pi|}G^{\beta}) (\frac {\gamma(x+w)-\gamma(x)}{w^2})  \prod_{B \in \pi} \frac {\partial^{|B|} \gamma(x+w) - \partial^{|B|} \gamma(x)} {w}\right) \\ \frac {(\partial^{l_2 +1}\gamma (x+\tau_1 w) - \partial^{l_2 +1}\gamma (x+\tau_2 w))(\partial^{l_3 +1}\gamma (x+\tau_1 w)  - \partial^{l_3 +1}\gamma (x+\tau_2 w))}{w^2} \partial^{l_4} \kappa (x+s_1 w)
\end{align*}
where $\Pi_{l_1}$ denotes the set of all partitions of the set $\{1, \ldots, l_1\}$.
Using the fundamental theorem of calculus this can be brought into the form
\begin{align*}
 \partial_x^k g^{2,2}_{s_1,\tau_1, \tau_2} (x,w) = 
 \sum_{l_1+l_2+l_3+l_4=k} \left( \sum_{\pi \in \Pi_{l_1}} (\partial^{|\pi|}G^{\beta}) (\frac {\gamma(x+w)-\gamma(x)}{w^2})  \prod_{B \in \pi} \int_0^1 \partial^{|B|+1}\gamma(x+s_B w)ds_B \right)\\ \frac {(\partial^{l_2 +1}\gamma (x+\tau_1 w) - \partial^{l_2 +1}\gamma (x+\tau_2 w))(\partial^{l_3 +1}\gamma (x+\tau_1 w) - \partial^{l_3 +1}\gamma (x+\tau_2 w))}{w^2} \partial^{l_4} \kappa (x+s_1 w)
\end{align*}
We choose $p_{B}= p_i =(\#\pi +4) p$  and observe that
$$
 |B|+1- \frac 1 {p_i} \leq |B|+1 \leq k+1 \leq k+2-\frac 1 2
$$
and
$$
 l_i + \frac 32 -\frac 1 2 \leq l_i + \frac 32 \leq  k+\frac 32  \leq k+2- \frac 1 2
$$
Hence, we can apply the Lemmata~\ref{lem:SmallW} and \ref{lem:BigW} to get an estimate as claimed for each of the summands with 
\begin{equation*}
 \theta \leq \frac{\left( l_1 + l_2 + l_3 + l_4 + 5-3  \right) - \frac 12 }{ \tilde k-\frac 32}.
\end{equation*}
Using the identity
\begin{equation*}
 P^\bot_{\gamma'}(R) = R - \langle R,\gamma' \rangle \gamma'
\end{equation*}
and treating the second term in this difference in the same way as above, we get the second estimate in the assertion.

\end{proof}

\begin{lemma} \label{lem:EstimateL2H}
 We have
 \begin{equation*}
  \|\partial_s^k \mathcal H - \partial_s ^k Q\|^2_{L^2(B_{2\beta}(0))} 
  \leq   (C_\varepsilon M_{\frac 3 2}^\alpha+  \varepsilon) ( S_{k+3}  + C_\varepsilon)
 \end{equation*}
 for suitable constants $\alpha>0$  and
 \begin{equation*}
  \|\partial_s^k \mathcal H\|_{L^2 (B_{2\beta} (0))} \leq (C_\varepsilon M_{\frac 3 2}^\alpha+  \varepsilon) S_{k+3}  + C_\varepsilon ( M_2^\beta +1).
 \end{equation*}
\end{lemma}

\begin{proof}
 We use 
 $$
  \mathcal H \gamma (x) = P_{ \gamma'}^\bot \tilde {\mathcal H} \gamma(x) =  \tilde{ \mathcal H } \gamma(x) - 
  \langle \tilde {\mathcal H} \gamma(x), \gamma'\rangle \gamma'
 $$
 together with the decomposition 
 $$
  \tilde {\mathcal H} = Q + R
 $$
 to write
 $$
  \mathcal H = Q - P^T_{\gamma'}Q + P^\bot_{\gamma'}R,
 $$
 where $P^T_{\gamma'}$ denotes the projection onto the tangential part. 
 
 Lemma~\ref{lem:EstimateL2TildeH} tells us that
 \begin{align*}
  \|\partial_s^k P_{\gamma'}^\bot R\|^2_{L^2}  & \leq C( M_{\frac 32}^{\alpha} + \varepsilon) S_{k+3} + C_\varepsilon.
 \end{align*}

 To deal with the term containing $Q$ we use
  $P_{\gamma'}^T Q = \langle Q\gamma, \gamma'\rangle \gamma'$ .
 Leibniz's rule yields 
 \begin{align*}
  \partial_s^k \left( \langle Q, \gamma' \rangle \gamma' \right) = \langle \partial_s^k Q, \gamma' \rangle \gamma' + I_1, \\
 \end{align*}
 where $I_1$ is a linear combination of terms
 \begin{equation*}
  Q[\partial_s^{k_1} \gamma] \partial _s^{k_2} \gamma'  \partial_s^{k_3} \gamma'
 \end{equation*}
 with $k_1,k_2,k_3 \in \mathbb N_0$, $k_1 +k_2 +k_3 =k$ and $k_2+k_3 \geq 1$. By H\"older's inequality the $L^2$-norm over $B_{2\beta}(0)$ of all these terms can be estimated by
 \begin{equation*}
   C \|Q\partial_s^{k_1} \gamma\|_{L^2(B_{2\beta})} \|\partial_s^k \gamma'\|_{L^4(B_{2\beta})}\| \partial_s^{k_3} \gamma'\|_{L^4(B_{2\beta})}
 \end{equation*}
As in the proof of Lemma~\ref{lem:STildeSComparable} we see that
\begin{align*}
 \|Q \partial_s^{k_1} \gamma \|_{L^2(B_{2 \beta}(0)} \leq C \|\partial_s^{3+k_1} \gamma\|_{L^2(B_{4 \beta} (0))} + \varepsilon S_{k_1+3} + C_\varepsilon.
\end{align*}
Hence, the interpolation estimates give
\begin{align*}
 I_1 \leq \varepsilon S_{k + \frac 72} + C_\varepsilon.
\end{align*}

 We now pick up the argument from the proof of Lemma~\ref{lem:PrincipalPart} to estimate the term
 \begin{equation*}
  \langle \partial_s^k Q \gamma, \gamma' \rangle \gamma'
 \end{equation*}
 Using the linearity of $Q$, we can rewrite
 \begin{align*}
  \langle \partial_s^k Q, \gamma' \rangle = \sum_{i=1}^n \langle \tilde Q [\partial_s^k \kappa_i] \gamma_i' - \tilde Q [\partial_s^k \kappa_i \gamma'_i].
 \end{align*}
 Form Lemma~\ref{lem:Commutator} we then get
 \begin{align*}
  \|\langle \partial_s^k Q \gamma,  \gamma'\rangle \gamma' \|_{L^2(B^2(0))} &\leq C \left(\|\partial_s^k\kappa\|_{B^{\frac 1 2 }_{4,2} (B_3(0))} \, \|\gamma'\|_{B^{\frac 1 2 }_{4,2}(B_3(0))}   + \sum_{j \in \mathbb N}\frac {\|\partial_s^k \kappa \|^2_{L^{4}}} {(\Lambda +j )^2} +1 \right)
  \\ & \quad+ C \|\partial_s^k\kappa\|_{L^2(B_3(0))} \,(\|\gamma'\|_{C^{0,1}(B_3(0))} +1) 
  \\
  &\leq (C_\varepsilon M_{3/2}^{\frac 1 2} + \varepsilon) S^{\frac 1 2}_{k+3} + C_\varepsilon 
 \end{align*}

\end{proof}

\subsubsection{Estimate of the highest order term}

\begin{lemma} \label{lem:HighestOrderTerm}
 If $M_{\frac 32 } \leq 1 $ and $ \gamma(0) \in B_2(0)$, we have
 \begin{equation*}
  -\int_{\mathbb R / l \mathbb Z} \left \langle \partial_s^{k+2} \mathcal H \gamma, \partial_s^k \kappa  \right \rangle \phi ds \leq  - \tilde M_{k+7/2} (0)
  + C M^\alpha_{\frac 3 2 } S_{k+\frac 7 2} (0).
 \end{equation*}
\end{lemma}

\begin{proof}
 The main strategy is, to use partial integration to move $1+\frac 1 2$ derivatives from the term $\partial_s^{k+2} \mathcal H$ to the term $\partial_s^k \kappa$. But before, we want to get rid of the projection onto the normal part contained in the definition of $\mathcal H$. We have
 \begin{equation*}
  \mathcal H \gamma (x) = P^\bot_{\gamma'(x)} (Q\gamma (x) + R \gamma(x)).
 \end{equation*}
 Let us first deal with the terms containing $R$. Integration by parts gives 
 \begin{align*}- \int_{\mathbb R / l \mathbb Z} \partial_s^{k+2} \left(P^\bot_{\gamma'(x)} (R \gamma (x)) \right) \partial_s^k \kappa \phi ds 
 & =  \int_{\mathbb R / l \mathbb Z} \partial_s^{k+1} \left(P^\bot_{\gamma'(x)} (R \gamma (x)) \right) \partial_s^{k+1} \kappa \phi ds \\
 & \quad   \int_{\mathbb R / l \mathbb Z} \partial_s^{k+1} \left(P^\bot_{\gamma'(x)} (R \gamma (x)) \right) \partial_s^{k} \kappa \phi' ds,
 \end{align*}
 which we can estimate using the product rule, H\"older's inequality and Lemma~\ref{lem:EstimateL2TildeH} by
 \begin{equation*}
  (C M_{\frac 32}^\alpha + \varepsilon ) S_{k+\frac 72} ^{\frac {2(k+1)+3}{4(k + 2}}
\|\partial_s^k \kappa\|_{L^2}
  \leq (C M_{\frac 32}^\alpha + \varepsilon ) S_{k+\frac 72} + C_\varepsilon.
 \end{equation*}
 So we get
 \begin{equation} \label{eq:Highest1}
  - \int_{\mathbb R / l \mathbb Z} \partial_s^{k+2} \left(P^\bot_{\gamma'(x)} (R \gamma (x)) \right) \partial_s^k \kappa \phi ds \leq (C M_{\frac 32}^\alpha + \varepsilon ) S_{k+\frac 72} + C_\varepsilon.
 \end{equation}

To estimate $\int_{\mathbb R / l \mathbb Z} \langle \partial_s^{k+2}P^\bot_{\gamma'} Q, \partial_s^k \kappa \rangle \phi ds$ we write
 $$
  P_{\gamma'(x)}^\bot Q \gamma(x) = Q \gamma (x) - \langle Q \gamma (x), \gamma'(x) \rangle \gamma'(x) = Q \gamma - P^T_{\gamma'} Q.
 $$
 Using
 \begin{align*}
  \langle Q \gamma(x), \gamma'(x) \rangle &= 2  \int_{-l}^l \int_{0}^1 (1-s) \frac {(\gamma'(x+sw)- \gamma'(x))( \kappa(x+sw) - \kappa(x))}{w^2} ds dw dx
 \end{align*}
 we get from the Lemmata~\ref{lem:SmallW} and \ref{lem:BigW} 
 \begin{align*}
   \|\partial_s^{k+1} P^T_{\gamma'} (Q)\|^2_{L^2(B_{2\beta}(0))} \leq C (M^\alpha_{\frac 3 2} + \varepsilon) S^{\frac {2k+3}{2(2k+4)}}_{k+\frac 72}.
 \end{align*}
 Hence, Cauchy's inequality implies
 \begin{align} \label{eq:Highest2}
  - \int_{\mathbb R / l \mathbb Z} \partial_s^{k+1}  P^T_{\gamma'} Q\gamma  \partial_s^{k+1} \kappa \phi ds
  \leq \varepsilon S_{k+ \frac 72} + C_\varepsilon.
 \end{align}
 The term $\int_{\mathbb R / l \mathbb Z} \langle \partial_s^{k+2} Q, \partial_s^k \kappa \rangle ds$ can be rewritten using \eqref{eq:NewFormulaForQ} as
 \begin{align*}
  \int_{\mathbb R / l \mathbb Z} & \langle \partial^{k+2}_s Q \gamma , \partial^k \kappa \phi \rangle ds
  = \int_{\mathbb R / l \mathbb Z} \langle Q \partial^{k} \kappa , \partial^k \kappa \rangle \phi ds
  \\ &=  2\int_{\mathbb R / l \mathbb Z} \int_{-l}^l \int_{0}^1 (1-s) \frac {|\partial_s ^{k+1} \kappa (x+sw)- \partial_s ^ {k+1} \kappa (x)|^2}{w^2} \phi(\gamma(x))ds dw dx
  \\ & \quad +  2\int_{\mathbb R / l \mathbb Z} \int_{-l}^l \int_{0}^1 (1-s) \frac {(\partial_s ^{k+1} \kappa (x+sw)- \partial_s ^ {k+1} \kappa (x)) (\phi(\gamma(x+sw)) - \phi(\gamma(x)))}{w^2} \partial_x^{k+1}\kappa(x+sw)ds dw dx.
 \end{align*}
 We observe that
 \begin{align*}
  \int_{\mathbb R / l \mathbb Z} \int_{-l}^l \int_{0}^1 (1-s) &  \frac {|\partial_s ^{k+1} \kappa (x+sw)- \partial_s ^ {k+1} \kappa (x)|^2}{w^2} \phi(x)ds dw dx 
  \\ &\geq \int_{B_1(0)} \int_{-l}^l \int_{0}^1 (1-s) \frac {|\partial_s ^{k+1} \kappa (x+sw)- \partial_s ^ {k+1} \kappa (x)|^2}{w^2} ds dw dx
  \\ & \geq   \int_{B_1(0)} \int_{0}^1 (1-s) s  \int_{-sl}^{sl}  \frac {|\partial_s ^{k+1} \kappa (x+\tilde w)- \partial_s ^ {k+1} \kappa (x)|^2}{\tilde w^2} d\tilde w  ds dx
  \\ & \geq c_0 \int_{B_1(0)} \int_{-\frac l2} ^{\frac l2 }  \frac {|\partial_s ^{k+1} \kappa (x+\tilde w)- \partial_s ^ {k+1} \kappa (x)|^2}{\tilde w^2} d\tilde w  dx \\ & \geq c_0 \tilde M_{k + \frac 72 } (0).
  \end{align*}
 Furthermore, we decompose
 \begin{align*}
  \Bigg| \int_{\mathbb R / l \mathbb Z} &\int_{-l}^l \int_{0}^1 (1-s) \frac {(\partial_s ^{k+1} \kappa (x+sw)- \partial_s ^ {k+1} \kappa (x)) (\phi(\gamma(x+sw)) - \phi(\gamma(x)))}{w^2} \partial_x^{k+1}\kappa(x+sw)ds dw dx \Bigg|
   \\ & \leq   \int_{B_{\Lambda}(0)} \int_{|w|\geq \Lambda} \int_{0}^1 (1-s) \frac {|\partial_s ^{k+1} \kappa (x+sw)- \partial_s ^ {k+1} \kappa (x)|}{w^2} |\partial_s^{k+1}\kappa(x+sw)|ds dw dx \\ & +  \int_{B_{\Lambda}(0)} \int_{|w|\leq \Lambda} \int_{0}^1 (1-s) \frac {|\partial_s ^{k+1} \kappa (x+sw)- \partial_s ^ {k+1} \kappa (x)|}{|w|}| \partial_s^{k+1}\kappa(x+sw)|ds dw dx 
 \end{align*}
 We use Lemma \ref{lem:SmallW} and Lemma \ref{lem:BigW} to estimate the first term and the second term by
 \begin{align*}
  S_{k+\frac 72 }^{\theta } M^{1-\theta}_{\frac 32} + M_{\frac 32}
 \end{align*}
 where $\theta = \frac {2k +3 + \frac 12  }{2k+4}<1$. Hence, Cauchy's inequality yields 
 \begin{align} \label{eq:Highest3}
  -\int_{\mathbb R / l \mathbb Z} & \langle \partial^{k+2}_s Q \gamma , \partial^k \kappa \phi \rangle ds \leq - c_0 M_{k+\frac 72} (0) + \varepsilon S_{k + \frac 72} + C_\varepsilon.
 \end{align}
 The inequalities \eqref{eq:Highest1}, \eqref{eq:Highest2}, and \eqref{eq:Highest3} prove the statement of the lemma.
\end{proof}

\begin{lemma} [Differential inequality for Energies of higher order] \label{lem:inequalityHigherEnergyCritical}
 For every $\varepsilon>0$ there is a constant $C_\varepsilon< \infty$ depending only on $\varepsilon, n$ and $k$ and $c_k >0$ such that
 \begin{equation*}
  \partial_t \int_{\mathbb R / \mathbb Z} |\partial^k_s \kappa|^2 \phi ds + c_k\tilde M_{k + \frac 72}
  \leq C ( M_{\frac 32 }^\alpha + \varepsilon)  S^{ext}_{k + \frac 72}(0) + C_\varepsilon.
 \end{equation*}
\end{lemma}

\begin{proof}
From \eqref{eq:EvolutionOfHigherOrderEnergies} we get
 \begin{multline}
  \partial_t \int_{\mathbb R / \mathbb Z} |\partial_s^{k} \kappa|^2  \phi ds
  = 2 \int_{\mathbb R / \mathbb Z}
  \langle \partial_s^{k+2}V,\partial_s^k \kappa\rangle \phi ds
 + 2\int
  \langle P^{k}_2(V,\kappa)  \tau, \partial^{k+1}_s
\kappa\rangle \phi ds
\\
 \quad + 2 \int \langle P_3^{k}(V,\kappa), \partial_s^{k} \kappa
\rangle \phi ds
- \int |\partial_s^k \kappa|^2 \langle \kappa ,  V\rangle \phi ds
+ \int_{\mathbb R / l \mathbb Z} |\partial^k_s \kappa| \nabla_{V} \phi ds
\end{multline}
Lemma~\ref{lem:HighestOrderTerm} gives
\begin{align*}
 2 \int_{\mathbb R / \mathbb Z}
  \langle \partial_s^{k+2}V,\partial_s^k \kappa\rangle \phi ds  \leq - \tilde M_{k + \frac 72} (0)
  +  (M_{\frac 3 2}^\alpha + \varepsilon) \varepsilon S_{k + \frac 32 } + C_\varepsilon  .
\end{align*}
Let $k_1 + k_2 = k$. H\"older's inequality, standard interpolation estimates and Lemma~\ref{lem:EstimateL2H} give
\begin{align*}
 \int_{\mathbb R / l \mathbb Z} \partial_s^{k_1} V \ast \partial_s^{k_2} \kappa \ast \tau \ast \partial_s^{k+1} \kappa \phi ds
 & \leq \|\partial_s^{k_1} V\|_{L^2} \|\partial_s^{k_2} \kappa\|_{L^4} \|\partial_s^{k+1} \kappa\|_{L^4}
 \\ &\leq C ( M_{\frac 32 }^\alpha + \varepsilon)  S_{k + \frac 72} + C_\varepsilon
\end{align*}
and hence
\begin{align*}
 \int
  \langle P^{k}_2(V,\kappa)  \tau, \partial^{k+1}_s
\kappa\rangle \phi ds \leq C ( M_{\frac 32 }^\alpha + \varepsilon)  S_{k + \frac 72} + C_\varepsilon.
\end{align*}

Similarly we get the estimate
\begin{equation*}
 \int \langle P_3^{k}(V,\kappa), \partial_s^{k} \kappa
\rangle \phi ds \leq C ( M_{\frac 32 }^\alpha + \varepsilon)  S_{k + \frac 72} + C_\varepsilon.
\end{equation*}
 and
\begin{equation*}
 \int_{\mathbb R / l \mathbb Z} |\partial^k_s \kappa|^2 \nabla_{V} \phi ds \leq \|\partial_s ^k \kappa \|^2_{L^4}  \|V\|_{L^2}
  \leq   (C_\varepsilon M_{\frac 32 }^\alpha + \varepsilon)  S_{k + \frac 72} + C_\varepsilon.
\end{equation*}
Hence, we have
\begin{align*}
 &2\int
  \langle P^{k}_2(V,\kappa)  \tau, \partial^{k+1}_s
\kappa\rangle \phi ds
  + 2 \int \langle P_3^{k}(V,\kappa), \partial_s^{k} \kappa
\rangle \phi ds
\\ &- \int |\partial_s^k \kappa|^2 \langle \kappa ,  V\rangle \phi ds
+ \int_{\mathbb R / l \mathbb Z} |\partial^k_s \kappa| \nabla_{V} \phi ds
\leq 5 \varepsilon S_{k + \frac 72} + C_\varepsilon \left( M_2 ^\beta +1  \right)
\end{align*}
Together, these estimates imply 
\begin{align*}
 \partial_t \int_{\mathbb R / \mathbb Z} |\partial^k_s \kappa|^2 \phi ds + \tilde M_{k+\frac 72}(0)
  \leq C ( M_{\frac 32 }^\alpha + \varepsilon)  S_{k + \frac 72} + C_\varepsilon.
\end{align*}
As $S_{k+\frac 72} \leq C S^{ext}_{k+\frac 72}$ we get the assertion.
\end{proof}

\subsubsection{Proof of Proposition\protect{~\ref{prop:estimateWillmore}}}

We get from Lemma~\ref{lem:inequalityHigherEnergyCritical}
\begin{equation*}
 \partial_t \int_{\mathbb R / \mathbb Z} |\kappa|^2  \phi ds + c_0 \tilde M_{\frac 72}(0) \leq  (C_\varepsilon M^{\alpha}_{\frac 3 2} + \varepsilon)  (S^{ext}_{\frac 7 2 }(x) + C_\varepsilon)
\end{equation*}
Integrating this inequality and using that  $M_{\frac 3 2 } \leq \varepsilon_0 < 1$ we get
\begin{equation} \label{eq:EstKappa1}
\begin{aligned}
 \int_{\mathbb R / l\mathbb Z} |\kappa_{\gamma_1}|^2 \phi ds &+ c_0\int_{\tau}^{1} \tilde M_{\frac 72}(0,t) d t  \\ &\leq \int_{\mathbb R / l\mathbb Z} |\kappa_{\gamma_{\tau}}|^2 \phi ds + (C_\varepsilon \varepsilon_0^\alpha + \varepsilon) \int_{\tau}^1 S^{ext}_{\frac 72} (x,t) dt + C_\varepsilon (1-\tau) \\
 & \leq \int_{\mathbb R / l\mathbb Z} |\kappa_{\gamma_{\tau}}|^2 ds + (C_\varepsilon \varepsilon_0^\alpha  + \varepsilon) \int_{\tau}^1 S^{ext}_{\frac 72} (x,t) dt + C_\varepsilon (1-\tau).
 \end{aligned}
\end{equation}
Integrating again over $\tau \in [0, \frac 1 2]$ yields
\begin{align*}
  c_0 \int_0^{\frac 1 2}\int_{\tau}^{1} & \tilde M_{\frac 72}(0,t) d t \\ & \leq \int_0^{\frac 1 2}\int_{\mathbb R / l\mathbb Z} |\kappa_{\gamma_{\tau}}|^2 \phi(\gamma)ds + (C_\varepsilon \varepsilon_0^\alpha + \varepsilon) \int_0^{\frac 1 2} \int_{\tau}^1 S^{ext}_{\frac 72} (x,t) dtd\tau  + C_\varepsilon (1-\tau) \\
 & \leq \int_0^{\frac 1 2}\int_{\mathbb R / l\mathbb Z} |\kappa_{\gamma_{\tau}}|^2 \phi(\gamma)ds + (C_\varepsilon \varepsilon_0^\alpha + \varepsilon) \int_0^{\frac 1 2} \int_{\tau}^1 S^{ext}_{\frac 72} (x,t) dt + C_\varepsilon .
\end{align*}

We can estimate the first term, using interpolation estimates as in Subsection~\ref{subsec:EstimatesEnergyDensity}, by
\begin{equation*}
 C \left(\int_0^{\frac 1 2}\tilde  S^{ext}_3 dt + 1 \right)  \leq  C (IM_3 + 1)
\end{equation*}
which is bounded by Proposition~\ref{prop:estimateLocalDensity}.
Assuming that 
\begin{equation*}
  IM_{\frac 72}=\sup_{x \in \mathbb R ^n} \int_0^{\frac 1 2}\int_{t_1}^{1} \tilde M_{\frac 72}(x,t) d \tau dt  = \int_0^{\frac 12}\int_{\tau }^{1}\tilde M_{\frac 72} (0,t)   dt
\end{equation*}
and using that
\begin{equation*}
  \sup_{x \in \mathbb R / l \mathbb Z} \int_0^{\frac 1 2}\int_{\tau}^1 S^{ext}_{\frac 72}(x,t) d \tau dt  = \int_0^{\frac 12}\int_{\tau}^{1}\tilde M_{\frac 72} (0,t)  d \tau dt \leq C IM_{\frac 72},
\end{equation*}
we deduce
\begin{equation*}
 c_0 IM_{\frac 72} \leq C +  (C_\varepsilon \varepsilon_0^\alpha + \varepsilon) IM_{\frac 72}+ C_\varepsilon.
\end{equation*}
Choose first $\varepsilon>0$ and then $\varepsilon_0>0$ sufficiently small, then we get
\begin{equation*}
IM_{\frac 72} \leq C.
\end{equation*}
Plugging this back into \eqref{eq:EstKappa1} we get the assertion

\subsection{Estimates for higher order energies}
It is tempting to just iterate the above argument to get control on higher order energies. Unfortunately, one would have to adapt $\varepsilon_1$ in each of the steps which would not yield to the desired result.
Instead we improve the differential estimate from the end of the last subsection assuming that $M_2$ is finite.
By literally the same argument as in the proof of the Lemmata in the last subsection but interpolating in all the arguments
between $W^{k+\frac 72,2}$ and $W^{2,2}$ instead of $W^{k+\frac 72,2}$ and $W^{\frac 32,2}$ we get

\begin{lemma} [Differential inequality for Energies of higher order] \label{lem:inequalityHigherEnergy}
 For every $\varepsilon>0$ there is a constant $C_\varepsilon$ depending only on $\varepsilon, n$ and $k$ and a constant $c_k >0$ such that
 \begin{equation*}
  \partial_t \int_{\mathbb R / \mathbb Z} |\partial^k_s \kappa|^2 \phi ds + c_k M_{k+\frac 72}(0)
  \leq \left( \varepsilon S^{ext}_{k + \frac 72 } + C_\varepsilon M_2^\beta \right)
 \end{equation*}
\end{lemma}

Now we are finally able to conclude the proof of the $\varepsilon$-regularity theorem. We prove inductively the following statement

\begin{proposition} \label{prop:HigherOrderEstimates}
 There is an $\varepsilon_0 >0$ and constant $C_k < \infty$ such that
 $$
  \sup_{x \in \mathbb R^n} E_{B_1(x)} (\gamma_0) \leq \varepsilon_0
 $$
 implies
 $$
  \sup_{x \in \mathbb R^n}\|\partial_s^k \kappa_{\gamma(t)} \|_{L^2(B_1(x))} \leq \frac C {t^{\frac k 3 - \frac 1 2}}.
 $$
\end{proposition}

\begin{proof}
 We prove by induction on $k$ that
 $$
  \|\partial_s^k \kappa_{\gamma(t)} \|_{L^2(B_1(x))} \leq \frac {C_k} {t^{\frac k 3 - \frac 1 2}}
 $$
 and
 $$
  \sup_{x \in \mathbb R^n}\int_{\frac t2}^{t} M_{k+\frac 72}(x,t) dt  \leq  \frac {C_k} {t^{\frac k 3 - \frac 3 2}} .
 $$
 Again by scaling properties of the solution it is enough to show these inequalities for $t=1$.
 Let us fix $\varepsilon_0>0$ such that we can apply Proposition~\ref{prop:estimateLocalDensity} and Proposition~\ref{prop:estimateWillmore}, i.e. such that the M\"obius energies on balls of radius 1 are small and the elastic energy on unit balls is bounded for times larger than $t=\frac 14 .$ Hence, the statement is true for $k=0$.
 
 Let us assume that we have the bound claimed for $k-1$ and let $t = \frac 12$.
 By Lemma~\ref{lem:inequalityHigherEnergy}  for every $\varepsilon >0$ we have
 \begin{equation} \label{eq:DifferentialInequality}
  \partial_t \int_{\mathbb R / \mathbb Z} |\partial^k_s \kappa|^2 \phi ds + c_k \tilde M_{\frac 72 +k}(x)
  \leq   \varepsilon S^{ext}_{k + \frac 72 } + C_\varepsilon.
 \end{equation}
 Integrating the inequality, we get for all $0<\tau <1$ that
 \begin{align*}
  \int_{\mathbb R / \mathbb Z} |\partial^k_s \kappa_1|^2 \phi ds &+ c_k \int_\tau ^1 \tilde M_{\frac 72 + k} (x,t) dt 
   \\ &\leq \int_{\mathbb R / \mathbb Z} |\partial^k_s \kappa_\tau|^2 \phi ds +  \varepsilon \int_\tau ^1 S^{ext}_{k + \frac 72 } (x,t) dt + C_\varepsilon (1-t)
 \end{align*}
 We integrate this inequality for $\tau \in [\frac 14 , \frac 12 ]$ to get
 \begin{multline*}
   \frac 1 4 \int_{\mathbb R / \mathbb Z} |\partial^k_s \kappa_1|^2 \phi ds + c_k \int_ {\frac 1 4}^{\frac 1 2} \int_\tau ^1 \tilde M_{\frac 72 + k} (x,t) dt d\tau \\  
  \leq \int_{\frac 1 4}^\frac 1 2\int_{\mathbb R / \mathbb Z} |\partial^k_s \kappa_\tau|^2 \phi ds  dt d \tau +  \varepsilon \int_{\frac 1 4} ^{\frac 12 }\int_\tau ^1 S^{ext}_{k + \frac 72 } (x,t) dt + C_\varepsilon. 
 \end{multline*}
 Since interpolation estimates yield
 \begin{equation*}
    \int_{\frac 1 4}^\frac 1 2\int_{\mathbb R / \mathbb Z} |\partial^k_s \kappa_\tau|^2 \phi ds  dt d \tau 
    \leq C \int_{\frac 1 4}^\frac 1 2 (S^{ext}_{k+\frac 32}+1) dt  \leq  C \sup_{x\in \mathbb R^n} \int_{\frac 1 4}^\frac 1 2 (M_{(k-1)+\frac 72}(x,t)+1) dt \leq C
 \end{equation*}
 by the induction hypotheses, we deduce that
 \begin{equation} \label{eq:IntegralEstimate}
   \frac 1 4 \int_{\mathbb R / \mathbb Z} |\partial^k_s \kappa_1|^2 \phi ds + c_k \int_ {\frac 1 4}^{\frac 1 2} \int_\tau ^1 \tilde M_{\frac 72 + k} (x,t) dt d\tau  
  \leq C +  \varepsilon \int_{\frac 1 4} ^{\frac 12 }\int_\tau ^1 S^{ext}_{k + \frac 72 } (x,t) dt + C_\varepsilon. 
 \end{equation}
 Let us now assume that the supremum 
 \begin{equation*}
   IM_{k+\frac 72} =\sup_{x\in \mathbb R^n}\int_{\frac 1 4} ^{\frac 12 }\int_\tau ^1 \tilde M _{k + \frac 72 } (x,t) dt
 \end{equation*}
 is attained in the point $x=0$. Since 
 \begin{equation*}
   \int_{\frac 1 4} ^{\frac 12 }\int_\tau ^1 S^{ext}_{k + \frac 72 } (x,t) dt \leq C IM_{k+\frac 72},
 \end{equation*}
 we deduce from \eqref{eq:IntegralEstimate}
 \begin{equation*}
  IM_{k+\frac 72} \leq \varepsilon IM_{k + \frac 72} + C_\varepsilon.
 \end{equation*}
 Choosing $\varepsilon >0$ small enough and absorbing, we get
 \begin{equation*}
   IM_{k+\frac 72} \leq C.
 \end{equation*}
 Hence, especially
 \begin{equation*}
  \int_{\frac 1 2 }^1 S^{ext}_{k + \frac 72 } (x,t) dt  \leq C \quad \forall x \in \mathbb R^n.
 \end{equation*}
Plugging this back into \eqref{eq:IntegralEstimate}, we derive
 \begin{equation*}
  \int_{\mathbb R / \mathbb Z} |\partial^k_s \kappa(s,1)|^2 \phi_x ds \leq  C \quad \forall x \in \mathbb R^n.
 \end{equation*}

\end{proof}

\begin{proof} [Proof of Theorem~\ref{thm:regularity}] Using scaled Sobolev embeddings we get the claimed estimates from Proposition~\ref{prop:HigherOrderEstimates} as long as the flow exists. So the only thing left is to show that $T >1$. But this follows by standard methods from the uniform estimates in Lemma~\ref{prop:HigherOrderEstimates}.
\end{proof}

\section{Applications}

\subsection{Blow-up profiles}

Using Theorem~\ref{thm:regularity}, we get the following classification of finite time blow-up
\begin{theorem} [Characterization of singularities]\label{blowup}
Let $\gamma \in C^\infty([0,T)\times \mathbb R / \mathbb Z,\mathbb R^n)$ be a maximal smooth solution of \eqref{eq:GFME}. There is a constant $\varepsilon_0 >0$ depending only on $n$ and $E(\gamma_0)$ such that if $T<\infty$ there are times $t_k \uparrow T$, points $x_k \in \mathbb R^n$ and radii $r_k \downarrow 0$ with
 $$
  E_{B_{r_k}(x_k)}(\gamma_{t_k}) \geq \varepsilon_0.
 $$
\end{theorem}

\begin{proof}
 Let us assume that $T< \infty$ and that there is an $r>0$ such that for all $t \in [0,T)$ 
 and all $x \in  \mathbb R^n$ we have
  \begin{equation*}
   E_{B_{r_j}(x_j)} (\gamma(t))
    \leq \varepsilon_0.
 \end{equation*}
 Then Theorem~\ref{thm:regularity} would tell us that $T >t_j + r_j^3 \rightarrow T+r^3$.
\end{proof}

Picking the concentration times more carefully, we can construct a blow-up profile at a singularity. As mentioned in the introduction, we localize the energy intrinsically
for this purpose, i.e. we work with $E^{int}_{B_r(x)}$ instead of $E_{B_r(x)}$. We do this for the simple reason that $E^{int}_{B_r(x)}$ is continuous in $r$ and $x$.

In the rest of this article we will expresse from time to time the integrals occurring as integrals over the image
\begin{equation*}
 \Gamma_t := \gamma(\mathbb R / l\mathbb Z,t).
\end{equation*}

\begin{theorem} [Blow-up profiles]\label{thm:ExistenceOfABlowupLimit}
There is an $\varepsilon_0 >0$ such that the following holds:
 Assume that $\gamma_t$ is a solution of \eqref{eq:GFME} that develops a singularity in finite time,
 i.e. $T<\infty$ and $r_j \rightarrow 0$. 
 Then there are points $x_j$ and times $t_j \rightarrow T$ such that
 $$
  E^{int}_{B_{r_j}(x_j)} (t_j) \geq \varepsilon_0.
 $$
 Let us now choose the points $x_j \in \mathbb R$ and times $t_j \in [0,T)$ such that
 \begin{equation*}
    \sup_{\tau \in [0,t_j], x \in \Gamma_\tau} E^{int}_{B_{r_j} (x)} (\gamma_{t_j})
    \leq E^{int}_{B_{r_j} (x)} (\gamma_{t_j}) = \varepsilon_0,
 \end{equation*}
and let $\tilde \gamma_j$ be re-parameterizations by arc-length of the rescaled and translated curves $$r_j^{-1} (\gamma_{t_j}-x_j)$$ such that
$\tilde \gamma_j(0) \in B_2(0)$. Then these curves 
sub-converge locally in $C^\infty$ to an embedded closed or open curve $\tilde \gamma_\infty: I \rightarrow \mathbb R^n$, $I = \mathbb R / l \mathbb Z$ of $I = \mathbb R$ resp.,  parameterized by arc-length. This curve satisfies
\begin{equation} \label{eq:ELBlowup}
 p.v \int_{\frac l2}^{\frac l2} \left( 2 \frac {P_\tau^\bot \left( \tilde \gamma(y)-\tilde \gamma(x) \right)}{| \gamma(y)- \gamma(x)|^2}
  - \kappa_\gamma (x) \right) \frac {dy}{|\gamma(y)-\gamma(x)|^2}  =0 \quad \quad \forall x \in  I, 
\end{equation}
and
\begin{equation*}
 E^{int}_{\overline B_1(0)} (\tilde \gamma_\infty)\geq \varepsilon_0.
\end{equation*}

\end{theorem}

\begin{proof}
 The first statement is an immediate consequence of Theorem~\ref{blowup} and the bi-Lipschitz estimate \eqref{eq:bilipschitz}.
 We consider the rescaled flows 
 \begin{equation*}
  \tilde \gamma^{(j)}(x,t) := \frac 1 {r_j} (\gamma(x,r_j^3 t + t_j) - x_j)
 \end{equation*}
 for $t \in (- \frac {t_j} {r_j^3}, 0]$
 which still solve Equation \eqref{eq:GFME}.
 Under the assumptions of the theorem we get
 \begin{equation*}
    E^{int}_{B_1 (0)} ( \tilde \gamma^{(j)}_{t}) \leq \varepsilon_0 \quad \forall t \in [-\frac {t_j }{r_j^3}, 0],
 \end{equation*}
 and hence from the bi-Lipschitz estimate
 \begin{equation*}
  E_{B_{\beta^{-1}}(0)} (\tilde \gamma_t^{(j)}) \leq \varepsilon_0 \quad \forall t \in [-\frac {t_j }{r_j^3}, 0] .
 \end{equation*}

 Hence we can apply Theorem~\ref{thm:regularity} to find  
 \begin{equation*}
  \|\partial_s^k \tilde \gamma_{t}\|_{C^k} \leq C_k
 \end{equation*}
 for all $k \in \mathbb N$ and $t \in [-\frac {t_j} {r_j^3} +1 ,0)$. As $-\frac {t_j} {r_j^3} \rightarrow - \infty$, we can use the theorem by Arzela-Ascoli to get, after going to a subsequence,
 \begin{equation*}
  \tilde \gamma_j \rightarrow \tilde \gamma
 \end{equation*}
 locally smoothly in time and space. Since all derivatives of $\gamma_\infty$ are uniformly bounded we furthermore deduce that 
 \begin{equation*}
  \mathcal H \tilde \gamma_\infty (x) = p.v. \int_{I}\left( 2 \frac {P_\tau^\bot \left( \tilde \gamma(y)-\tilde \gamma(x) \right)}{| \gamma(y)- \gamma(x)|^2}
  - \kappa_\gamma (x) \right) \frac {dy}{|\gamma(y)-\gamma(x)|^2} 
 \end{equation*}
 is well defined. Furthermore, we have 
 \begin{equation} \label{eq:Conv}
  \int_{-\delta_0}^0 \int_{\mathbb R / l_t \mathbb Z} |\mathcal H (\gamma^{(j)}_t)(x)|^2 dx dt = E(\gamma_{t_j - r_j^3}) - E(\gamma_{t_j})
  \rightarrow 0
 \end{equation}
for some subsequence $j$ and hence after going to a subsequence
\begin{equation*}
 \mathcal H \tilde \gamma^{(j)} (x) \rightarrow 0  
\end{equation*}
pointwise almost everywhere. 
We now show that
\begin{align*}
 \mathcal {H} \tilde \gamma_j \rightarrow \mathcal H \tilde \gamma_\infty
\end{align*}
pointwise.  For this purpose we again use the decomposition
\begin{equation*}
 \tilde {\mathcal H \gamma} = Q \gamma + R_1 \gamma + R_2 \gamma.
\end{equation*}
As
\begin{align*}
 \frac {|w|^\alpha} {|\gamma(x+w)-\gamma(x)|^{2+\alpha}} - \frac 1 {w^2} & \leq C \frac{\int_0^1 \int_0^1 |\gamma'(x+s_1 w) - \gamma'(x+s_2 w)|^2 ds_1 ds_2}{|w|^2} \\ &\leq C \min\{\|\kappa\|_{L^\infty(B_R(x)}, \frac 1 {|w|^2}\}
\end{align*}
we get that the integrands of both, $R_1(\gamma_j)$ and $R_2(\gamma_j)$, are uniformly bounded. As all the integrands also converge pointwise to the integrands of $R_1(\gamma_\infty)$ and $R_2(\gamma_\infty)$, the dominant convergence theorem yields
\begin{equation*}
 R(\gamma_j) \rightarrow R(\gamma_\infty).
\end{equation*}
For the integrand of $Q$ we use Taylor's approximation up to order $2$ to get
\begin{align*}
 \frac{\gamma(x+w)- \gamma(x) - w \gamma'(x)-\frac 12 w^2  \gamma''(x)}{w^4} =  \frac{\int_0^1 (1-s)^2 \gamma'''(x+sw) ds } {w}
\end{align*}
and write
\begin{equation*}
 Q \gamma = \int_{\mathbb R / l\mathbb Z} I dw,
\end{equation*}
where
\begin{equation*}
 \begin{cases}
   I(x,w) := \frac{\int_0^1 (1-s)^2 \gamma'''(x+sw) - \gamma'''(x) ds } {2w} &\text{ for } |w| \leq 1 \\
   I(x,w) := \frac{\int_0^1 (1-s) \gamma''(x+sw) - \gamma''(x) ds}{|w|^2} & \text{ else .}
 \end{cases}
\end{equation*}
The mean value theorem tells us that $|I(x,w)| \leq C\|\gamma''''\|_{L^\infty(B_1(x))}$ if $|w| \leq 1$, and $I(x,w) \leq  w^{-2} \|\gamma''\|_{L^\infty} \}$ else.
We get using the dominated convergence theorem
$Q \gamma_j \rightarrow Q \gamma_\infty$. This completes the proof of 
$$
 \mathcal H \tilde \gamma^{(j)} \rightarrow \mathcal H \tilde \gamma_\infty
$$
pointwise.

We get in view of \eqref{eq:Conv}
\begin{equation*}
 \int_{-\delta_0}^0 \int_\Gamma |\mathcal H \tilde \gamma_\infty |^2 d\mathcal H^1(x) dt \leq \lim_{j \rightarrow \infty} \int_{-\delta_0}^0 \int_{\mathbb R / l_t \mathbb Z} |\mathcal H (\gamma^{(j)}_t)(x)|^2 dx dt =0.
\end{equation*}
Since $\tilde \gamma_\infty$ is smooth, we obtain $\mathcal H \tilde \gamma_\infty \equiv 0$.
Furthermore, the local smooth convergence together with $E^{(int)}_{B_1 (0)} (\tilde \Gamma_j) = \varepsilon_0$
implies
\begin{equation*}
 E^{(int)}_{ \overline B_1 (0 )} (\tilde \Gamma_\infty) \geq \varepsilon_0.
\end{equation*}
\end{proof}

Using the evolutionary attractivity of critical points proven in \cite{Blatt2011c} we can further show
that the blow-up profile cannot be compact.

\begin{proposition}[Blow-ups profiles are never compact] \label{prop:BlowupsAreNoncompact}
The blow-up profile constructed in Theorem~\ref{thm:ExistenceOfABlowupLimit} cannot be compact.
\end{proposition}

\begin{proof}
 Let us assume that $\tilde \gamma_\infty$ was compact, i.e. that $\tilde \gamma_\infty \in C^{\infty} (\mathbb R / l \mathbb Z, \mathbb R^n)$ for suitable $l$. Then there would be a subsequence of $\tilde \gamma_j$ converging smoothly to the critical point $\tilde \gamma_\infty$ of $E$. Since furthermore $E(\gamma_t)\geq E(\tilde \gamma_\infty)$, we get from \cite[Theorem~1.5]{Blatt2011c}
 for all $t \in [0,T)$, that for $j$ large enough the flow $\tilde \gamma_{t}$ exists for all time
 and converges to a stationary point of $E$ - which is contradicting the assumption $T < \infty$.
\end{proof}

\subsection{Planar curves}

For a regular curve $\gamma$ the \emph{curvature vector} $\kappa$ is given by
\begin{equation*}
 \kappa = \frac {\gamma''}{|\gamma'|^2} - \frac {\left\langle \gamma'', \gamma' \right\rangle}{|\gamma'|^4} \gamma'
\end{equation*}
which is equal to $\gamma''$ if $\gamma$ is parameterized by arc-length.

Given two points $x, y \in I$ there is either a unique circle or a straight line -- which we like to think of as a degenerate circle -- 
going through $\gamma(x)$ and $\gamma(y)$ and being tangent to $\gamma$ at $x$.
Note that this is the same circle used to define 
the integral tangent-point energies.
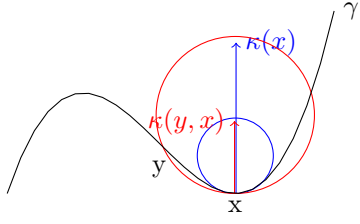
\begin{figure}
 \begin{tikzpicture}[domain=-2:2.3]
  \path[draw=blue] (1,-0.166) circle (.5) ;
  \draw[->, blue] (1.01,-0.666) -- (1.01, 1.333) node[right] {$\kappa(x)$};
  \path[draw=red] (1, .375) circle (1.041) ;
  \draw[->,red] (0.99,-0.666) -- (0.99, 0.294) node[left] {$\kappa(y,x)$};
  \draw[color=black] plot (\x,{\x*\x*\x/3-\x}) node[right] {$\gamma$};
  \node (a) at (1,-5/6) {x};
  \node (b) at (0,-1/3) {y};
 \end{tikzpicture}
 \caption{This picture shows the two circles playing a role in the geometric interpretation of the Euler-Lagrange equation of the M\"obius energy: The blue
 circle is the osculating circle at $x$ while the red circle is the circle going through $x$ and $y$ and being tangent to $\Gamma$ at $x$.}
\end{figure}
We denote by $\kappa_{\Gamma}(x,y)$ the curvature vector of this circle in $x$ and set $\kappa_{\Gamma}(x,y)=0$ if the tangent on $\Gamma$ in $x$ is
pointing in the direction of $y$ -- which is the curvature of the straight line.

\begin{lemma} \label{lem:TangentPointCurvature}
 We have
 \begin{equation*}
  \kappa_\gamma (x,y) = 2\frac {P^\bot_{\gamma'(x)} (\gamma(y)-\gamma(x))}{|\gamma(x)-\gamma(y)|^2}
 \end{equation*}
\end{lemma}

\begin{proof}
If the vectors $\gamma'$ and $\gamma(x) -\gamma(y)$ are co-linear, both sides of the identity obviously vanish. So we can assume that $P^\bot_{\gamma'(x)} ( \gamma(y) - \gamma(x)) \not=0$.
The circle going through $\gamma(x)$ that is tangential to $\gamma$ in the point $x$  with curvature vector $\kappa = a P^\bot (\gamma(y) -\gamma(x))$ is the set of all points $z \in \mathbb R^n$ satisfying 
\begin{equation*}
 |z-\gamma(x) - \frac {\kappa} {|\kappa|^2}|^2 = \frac 1 {|\kappa|^2}. 
\end{equation*}
This circle contains $\gamma(y)$ if and only if
\begin{equation*}
  |\gamma(y)-\gamma(x)|^2 = 2 \frac {\langle \kappa , \gamma(x) -\gamma(y) \rangle}{\kappa^2} =  \frac 2 {a}. 
\end{equation*}
Thus, $a = \frac 2 {|\gamma(x) - \gamma(y)|^2}$ which proves the lemma.

\end{proof}

Using Lemma~\ref{lem:TangentPointCurvature} we immediately get the following geometric interpretation of the Euler-Lagrange equation
\eqref{eq:ELBlowup}
\begin{lemma} [Geometric interpretation of the Euler-Lagrange equation]
 The curve $\gamma$ parameterized by arc-length satisfies $\mathcal H \gamma \equiv 0$ if and only if
 \begin{equation} \label{eq:ELGeometric}
  \lim_{\varepsilon \downarrow 0}\int_{\Gamma \setminus B_\varepsilon(x)} 
  \frac {\kappa_\gamma (x,y) - \kappa_\gamma (x)}{|x-y|^2} d \mathcal H^1 (y) =0
 \end{equation}
  for all $x \in I$.
\end{lemma}

In codimension one,  \eqref{eq:ELGeometric} is equivalent to 
\begin{equation} \label{eq:ELGeometricPlanar}
  \lim_{\varepsilon \searrow 0} \int_{I / B_\varepsilon (x)} 
  \frac {\langle \kappa_\gamma (x,y) - \kappa_\gamma (x), n (x) \rangle}{|x-y|^2} d \mathcal H^1 (y) =0
 \end{equation}
where $n$ is a unit normal along $\gamma$
We are now looking for situation that imply that the integrand on the left hand side of \eqref{eq:ELGeometricPlanar} has a sign and thus must vanish identically.
For $x \in I$ in which the curvature of $\gamma$ does not vanish, we denote by $OB(x)$ the open ball whose boundary is the osculating circle
along $\gamma$ in $x$, i.e.
\begin{equation*}
 OB(x) := B_{\frac 1 {|\kappa(x)|}} \left(\gamma(x)+ \frac {\kappa}{|\kappa|^2} \right).
\end{equation*}

\begin{lemma} \label{lem:GradientHasASign}
 If there is a point $x \in  I$ such that 
 \begin{equation*}
  OB(x) \cap \gamma (I) = \emptyset
 \end{equation*}
 or 
 \begin{equation*}
  \gamma(I) \subset \overline {OB(x)}
 \end{equation*}
 then
 \begin{equation*}
  \Gamma = \partial OB(x),
 \end{equation*}
 i.e. $\Gamma$ is a circle.
\end{lemma}

\begin{proof}
 If $\Gamma \cap OB(x) = \emptyset$, we get
 \begin{equation*}
  \left\langle \kappa_\Gamma(x,y),n(x) \right\rangle \leq \left\langle \kappa_\Gamma (x), n(x)\right\rangle,
 \end{equation*}
 and if $\Gamma \subset \overline {OB(x)}$
 \begin{equation*}
  \left\langle \kappa_\Gamma(x,y),n(x) \right\rangle \geq \left\langle \kappa_\Gamma (x), n(x)\right\rangle.
 \end{equation*}
 So in both cases 
 \begin{equation*}
  \left\langle \kappa_\Gamma(x,y),n(x) \right\rangle - \left\langle \kappa_\Gamma (x), n(x)\right\rangle
 \end{equation*}
 has a sign that is independent of  $y \in \Gamma$.
 
 Since $\mathcal H \gamma \equiv 0$ implies
 \begin{equation*}
  \lim_{\varepsilon \searrow 0} \int_{\Gamma / B_\varepsilon (x)} \frac {\left\langle (\kappa_\gamma (x,y) - \kappa_\gamma (x)), n(x) \right\rangle}{|\gamma(y)-\gamma(x)|^2} 
  d \mathcal H^1 (y) = 0
 \end{equation*}
 and the integrand has a sign, we get
 \begin{equation*}
  \left\langle (\kappa_\gamma (x,y) - \kappa_\gamma(x)), n(x) \right\rangle = 0
 \end{equation*}
 for all $y \in \Gamma$. But this implies
 \begin{equation*}
  \kappa_\gamma (x,y) = \kappa_\gamma (x)
 \end{equation*}
 for all $y \in \Gamma$ which by the definition of $\kappa_\Gamma (x,y)$ implies that
 \begin{equation*}
  y \in \partial OB(x).
 \end{equation*}
\end{proof}

\begin{theorem} \label{thm:planarsolutions}
 Let $\Gamma: I \rightarrow \mathbb R^2$ be a properly embedded smooth curve parameterized by arc-length satisfying
 \begin{equation*}
  p.v. \int_{I}
    \frac {\kappa_\gamma (x,y)- \kappa_\gamma(x)}{|\gamma(y)-\gamma(x)|^2} dy = 0.
 \end{equation*}
 Then $\gamma$ is either a straight line or a circle.
\end{theorem}

\begin{proof}
 Let us assume that $\gamma$ is not a straight line.
 We will show that then there is a point $x \in I$ with $\kappa(x) \not=0$ and 
 \begin{equation*}
  OB(x) \cap \gamma(I) = \emptyset ,
 \end{equation*}
 where $OB(x)$ is the open ball surrounded by the osculating circle on $\gamma$ at $x$, i.e.
 \begin{equation*}
  OB(x) := \left\{y \in \mathbb R^2 : \left|y- \left(\gamma(x)+ \frac {\kappa (x)}{|\kappa (x)|^2}\right) \right| \leq \frac 1 {|\kappa (x)|} \right\}
 \end{equation*}
 Then the statement follows from Lemma~\ref{lem:GradientHasASign}.
 We construct this point as follows: As $\Gamma= \gamma(I)$ is not a straight line, we find a point $x_1 \in \Gamma$
 with $\kappa_\Gamma (x_1) \not = 0$. Let $n$ be the continuous unit normal field pointing in the
 direction of $\kappa_\Gamma (x_1)$ at the point $x_1$.
 Then either $OB(x_1) \cap \Gamma = \emptyset$ in 
 which case we set $x=x_1$.
 If on the other hand $OB(x_1) \cap \Gamma \not= \emptyset$, there is a ball $B_1 \subset OB(x_1)$ touching
 $\Gamma$ in $x_1$ and at least one other point. Let $x_1'$ be one of these touching points nearest
 to $x_1$ and let $\Gamma_1$ denote the closed curve consisting of the arc of $\Gamma$
 between $x_1$ and $x_1 '$ and the part of the boundary of $B_1$ that makes this curve $C^1$
 and let $\Omega_1$ be the open set bounded by this curve.
 
 We now start an iterative scheme in order to find the desired point $x$.
 So let $x_2 \in \Gamma$ be the point on the part of the curve between $x_1$ and $x_1'$ which 
 divides this arc into two parts of equal length. Note that $x_2 \notin \overline{B_1}$.
 We choose 
 \begin{equation*}
  r_2 := \sup\{r: B_{r}\left (x_2 + \frac 1 rn(x_2) \right) \subset \Omega_1\}
 \end{equation*}
 Then either $B_2 = B_{r_2} (x_2 + \frac 1 {r_2} n_2)$ touches $\Gamma$ in $x$ up to second order
 and we set $x=x_2$ and have found our point $x$. Or we can chose
 $x_2' \in \Gamma_1$ to be one of the nearest points on $\Gamma_1$ touching 
 $B_2 = B_{r_2} (x_2 + \frac 1 {r_2} n_2)$. But then $x_2'$  must belong to the arc
 of $\Gamma$ between $x_1$ and $x_1 '$ since else $B_2$ touches $B_1$ from within 
 and hence $B_2 \subset B_1$ - which is not possible, as $x_2 \in \overline B_2$
 but $x_2 \not\in \overline B_1$. Hence, 
  \begin{equation} \label{eq:nestedIntervals}
    d_\Gamma (x_2,x_2') \leq \frac 1 2 d_\Gamma (x_1,x_1').
 \end{equation}
 
 Then we repeat the construction above, and either get our point $x$ in a finite number of steps, or get a sequence of points 
 $x_i$, $x_i'$ and balls $B_i \cap \Gamma = \emptyset$ such that
 $B_i$ touches $\Gamma$ in $x_i, x_i'$, the intervals $x_i, x_i'$
 are nested and the diameter of the balls $B_i$ is bounded by the 
 diameter of $\Gamma_1$ and from below by
 \begin{equation*}
   \|\kappa_\Gamma |_{[x_1, x_1']}\|_{L^\infty}^{-1} > 0.
 \end{equation*}

 In the latter case, there is a point $x \in \Gamma$ with
 \begin{equation*}
  x= \lim_{i \rightarrow \infty}x_i = \lim_{i \rightarrow \infty } x_i '
 \end{equation*}
 and it is well known that
 \begin{equation*}
  r_i \rightarrow \frac 1 {|\kappa (x)|}.
 \end{equation*}
We get for every $r< \frac 1 {|\kappa (x)|}$ that
 \begin{equation*}
  B_{r} (x + \frac 1 r n(x)) \subset B_n
 \end{equation*}
 for $n$ large enough. Hence,
 \begin{equation*}
  B_{r} \cap \Gamma = \emptyset \quad \forall r < \frac 1 {\kappa (x)}
 \end{equation*}
 which implies
 \begin{equation*}
  OB(x) \cap \Gamma = \emptyset.
 \end{equation*}
\end{proof}

Using the characterization of the solutions to \eqref{eq:GFME} we can now show

\begin{theorem}[The evolution of planar curves] \label{thm:GFMEplanar} Let $\gamma_0 \subset \mathbb R^2$ be a closed smoothly embedded curve. Then the negative gradient flow of the M\"obius energy exists for all times and converges to a round circle as time goes to infinity.
\end{theorem}

\begin{proof}
 Let us first prove the long time existence of the flow. Assume that a singularity occurs 
 after finite time. Then we construct a blow-up profile $\tilde \gamma_\infty$ as described in Theorem~\ref{thm:regularity}.
 But Theorem~\ref{thm:planarsolutions} implies that this blow-up must be a circle or straight line - which is not possible due to
 Proposition~\ref{prop:BlowupsAreNoncompact} and $E^{int}_{B_1(0)} (\tilde \gamma_\infty) \not= 0$.
 
 To prove the statement about the asymptotic behavior of the flow, we  let for $t \in (0,\infty)$ and $\varepsilon_0 >\varepsilon >0$ small enough the radius $r_t>0$ and $x_t \in \gamma_t$ be  such that
 \begin{equation*}
  E_{B^{int}_{r_t}(x_t)}(\gamma_t) = \sup_{x\in \gamma_t}  E_{B^{int}_{r_t}(x)}(\gamma_t) = \varepsilon.
 \end{equation*}
 Let us assume that 
 \begin{equation} \label{eq:growthofconcentrationradius}
  M:=\liminf_{t\in[0,\infty)} \frac{r_{t+\frac 1 2 r_t^3}}{ r_t} < \infty.
 \end{equation}
 Then we can choose a sequence $t_j \rightarrow \infty$
 such that 
 $$
  r_{t_j+ \frac 1 2 r_{t_j}^3} \leq 2M r_{t_j}
 $$
 As in Theorem~\ref{thm:ExistenceOfABlowupLimit}, let $\tilde \gamma_j$ be re-parameterizations of the
 rescaled curves $$ \frac 1 {r_j} \left\{ \gamma_{t_j + \frac 1 2 r_{t_j}^3} - x_{t_j + \frac 1 2 r_{t_j}^3} \right\}$$ by arc length such that $\tilde \gamma_j (0) =0$.
 Then these curves $\gamma_j$ sub-converge locally smoothly to a curve $\gamma_{\infty}$
 satisfying $\mathcal H \gamma_\infty \equiv 0$ which is not a straight line.
 Hence, due to Theorem~\ref{thm:planarsolutions} 
 $\gamma_\infty$ is a circle. Since $\tilde \gamma_{t_j} \rightarrow \gamma_\infty$ smoothly we get that for $j$ large enough,
 the flow starting with $\tilde \gamma_ j$ converges smoothly to a circle as times goes to infinity.
 Hence, the same is true for $\gamma_t$.

Let us assume that \eqref{eq:growthofconcentrationradius} was wrong  and let $L_t$ denote the length of the curve $\gamma_t$.  Then for every $\Lambda >0$ there is a $t_0$ such that $r_{t+\frac 1 2 r_t^3} \geq \Lambda r_t$ for all $t \geq t_0$. We iteratively define $t_{j+1} := t_j + \frac 1 2 r^3_j$  where $r_j := r(t_j)$ and get
 \begin{equation} \label{eq:rjLarge}
  r_{j} \geq \Lambda ^j r_{t_0}.
 \end{equation}
  Scaling our a priori estimates in Theorem~\ref{thm:regularity} we obtain 
  \begin{equation*}
   |\mathcal H \gamma_{\tilde t}  | \leq \frac  C {(\tilde t -t)^{\frac 23}r_t^{2}}
  \end{equation*}
 for all times $\tilde t\in t+(0 ,r_t^3)$ and hence
 \begin{equation*}
  \left|\frac d {dt}|_{t=\tilde t} L_t \right| \leq 2 \sup \left|\frac d {dt}|_{t= \tilde t} \gamma \right| \leq  \frac  C {(\tilde t -t)^{\frac 23}r_t^{2}}.
 \end{equation*}
 Integrating this inequality we obtain
 \begin{equation*}
  L_{t+\frac 1 2 r_t^3} \leq C \frac {r_t^{3}}{ r_t^{^2}} + L_t \leq C L_t.
 \end{equation*}
 Hence,
 $$
  L_{t_j+1} \leq C L_{t_j}
 $$
 and thus
 $$
  r_j \leq L_{t_j} \leq C^{j}L_{t_0}
 $$
 which contradicts \eqref{eq:rjLarge} for $\Lambda > C$ and $j$ large enough.

\end{proof}

\begin{appendix}

 \section{Besov spaces, commutator estimates and interpolation inequalities}
 
 For the convenience of the reader, let us gather some well-known and not so well-known facts about Besov-spaces in this section. We will stick to the notation used in \cite{Triebel1983} and 
 will assume that the reader in familiar with the definition of the Besov-spaces $B^s_{p,q}(\mathbb R^n)$ on $\mathbb R^n$ and the respective spaces $B^s_{p,q}(\Omega)$ on smooth domains $\Omega \subset \mathbb R^n$ as defined in Section 2.3.1 and Section 3.2.2 of \cite{Triebel1983}.
 
 Essential for our analysis is the following characterization of these spaces using finite differences. For an arbitrary function $f:\mathbb R^n \rightarrow \mathbb R$ these are inductively defined by
 \begin{equation*}
  (\Delta ^1_h f) (x):= f(x+h)-f(x) , \quad \quad (\Delta_h^lf) = \Delta^1_h \Delta^{l-1}_h f \text{ for } l=2,3, \ldots.
 \end{equation*}
 Furthermore, for a set $\Omega \subset \mathbb R^n$ we set $\Omega_{h,l}= \bigcap_{j=0}^l\{x \in \Omega: x+jh \in \Omega\}$.

 \begin{lemma}[Equivalent norms, cf. \protect{\cite[Section 2.5.12, 3.4.2, and 2.5.10 ]{Triebel1983}}] \label{lem:EquivalentNormsI}
 The following estimates hold:    
  \begin{enumerate}
   \item For $0 < p,q \leq \infty$, $s>\tilde \sigma_p := n\left(\frac 1 {\min\{p,1\}}-1\right)$. If $M > s$ and $M$ an integer, then
  \begin{equation*}
   \|f |B^{s}_{p,q}(\mathbb R^n)\|^{(2)}_M :=  \|f\|_{L^p(\mathbb R^n)} + \left( \int_{\mathbb R ^n}  \frac {\|\Delta_h^M f| L^p(\mathbb R^n \|^q}{|h|^{n+sq}} dh  \right)^{ \frac 1q }
  \end{equation*}
  is an equivalent quasi-norm on $B^s_{p,q}(\mathbb R^n)$.
  \item If $\Omega \subset \mathbb R ^n$ is a smooth domain $1 < p < \infty$, $s >0$
  and $k,l$ integers with $0 \leq k < s$ and $s< l+k$, then
  \begin{multline*}
   \|f |B^{s}_{p,q}(\Omega)\|^{(2)}   :=  \|f|L^p(\Omega)\| \\ + \sum_{|\alpha| \leq k}\Bigg\| \left( \int_{\mathbb R ^n}  \left( \int_{\Omega_{h,l}}|\Delta_h^l \partial^\alpha f(x)|^q dx \right)^{\frac qp} \frac {dh}{|h|^{n+sq}}   \right)^{ \frac 1q } \Bigg| L^p (\Omega)\Bigg\|
  \end{multline*}
  is an equivalent quasi-norm on $B^s_{p,q}(\Omega)$.
  \end{enumerate}
 \end{lemma}

 As an easy consequence, we get
 \begin{lemma} \label{lem:EquivalentNormsII}
  For $1 < p,q < \infty$ and $1> s >0$ we have
  \begin{equation*}
    \left( \int_{B_1(0)}  \frac {\|\Delta_h^M| L^p(B_1(0)) \|^q}{|h|^{n+sq}} dh  \right)^{ \frac 1q } \leq C\|f | B^s_{p.q}(B_2(0))\|
  \end{equation*}
 and
 \begin{equation*}
    \|f | B^s_{p.q}(B_1(0))\| \leq C \left( \int_{B_2(0)}  \frac {\|\Delta_h^M| L^p(B_2(0)) \|^q}{|h|^{n+sq}} dh  \right)^{ \frac 1q } 
  \end{equation*}
 \end{lemma}
 \begin{proof}
  From the definition of the norm, we deduce that there is an extension $\tilde f$ of $f|B_2(0)$
  such that
  \begin{equation*}
   \|f\|_{B^s_{p,q}(B_1(0))} \leq \|\tilde f\|_{B^s_{p,q}(\mathbb R^n)} \leq 2 \|f\|_{B^s_{p,q}(B_2(0))}.
  \end{equation*}
 Lemma~\ref{lem:EquivalentNormsI} gives
 \begin{align*}
  \left( \int_{B_1(0)}  \frac {\|\Delta_h^M| L^p(B_1(0)) \|^q}{|h|^{n+sq}} dh  \right)^{ \frac 1q }
  &\leq \|\tilde f|L^p(\mathbb R^n)\| + \left( \int_{\mathbb R ^n}  \frac {\|\Delta_h^M \tilde f| L^p(\mathbb R^n \|^q}{|h|^{n+sq}} dh  \right)^{ \frac 1q }
  \\ &\leq C \|\tilde f| B^s_{p,q} (\mathbb R^n) \| \leq C \|f | B^s_{p,q}(B_2(0))\|.
 \end{align*}
 
To get the second estimate, we extend $f|B_1(0)$ to a function $\tilde f$ such that 
\begin{equation*}
 \|f\|^M_{B^s_{p,q}(B_1(0))} \leq \|\tilde f\|_{B^s_{p,q}(\mathbb R^n)} \leq 2 \|f\|^M_{B^s_{p,q}(B_2(0))}.
\end{equation*}
and argue as above.
 \end{proof}

 We will now state the following interpolation inequalities in Besov-space. Since it seems to be hard to find a proof of this result in the literature, we include a proof here for the sake of completeness.
 
 \begin{lemma}[Interpolation inequalities] \label{lem:Interpolation}
 
  Let $\Omega \subset \mathbb R^n$ be a smooth bounded domain. If $0 \leq s_1< s_2 < s_3 $ and $p \in [2,\infty)$ satisfy $s_0 - \frac n2 < s_1 - \frac n {p} < s_w - \frac n 2$ then 
  \begin{equation*}
   \|f| B^{s_1}_{p_1, q}(\Omega)\| \leq C\|f| B^{s_0}_{2,2}(\Omega) \|^{1-\theta} \, \|f| B^{s_2}_{2,2}(\Omega)\|^\theta
  \end{equation*}
  for all $q \in [1,\infty]$ with $C=C(s,p,q,\Omega)$ and $\theta = \frac {\left( s_1 - \frac n p \right) - \left( s_0 - \frac n 2 \right)}{s_2 - s_0}$.
 \end{lemma}
 
 \begin{proof}
  By \cite[Proposition~1.3.2]{Lunardi1995} we have to show that the real interpolation space
  \begin{equation*}
   (B^{s_0}_{2,2}(\Omega),B^{s_0}_{2,2}(\Omega))_{\theta,1}
  \end{equation*}
 is continuously embedded in $B^{s_1}_{p,q}$. But this is indeed the case, as
  \begin{equation*}
   (B^{s_0}_{2,2}(\Omega),B^{s_0}_{2,2}(\Omega))_{\theta,1}
   = B^{\tilde s}_{2,1} (\Omega)
  \end{equation*}
  with $\tilde s = (1-\theta) s_0 + \theta s_2= s_1 + \frac n2 - \frac np >s_1$ by \cite[p. 204]{Triebel1992} and the Sobolev-embedding for Besov spaces \cite[p. 196]{Triebel1992} tells us that $B^{\tilde s}_{2,1}(\Omega)$ is continuously embedded into 
  $ B^{s_1}_{p_1,1}(\Omega) \subset B^{s_1}_{p_1,q}(\Omega) $ for all $q \in [1,\infty]$.
 \end{proof}

  One of the most important tools in this article is the following commutator estimate. 

 \begin{lemma} [Commutator estimates] \label{lem:Commutator}
  For $f,g \in C^\infty([-\Lambda, \Lambda])$ we have
  \begin{multline*}
   \|\tilde Q[fg] - g \tilde Q[f]- f\tilde Q[g]\|_{L^r (B_{1/2}(0))}  \\ \leq C \Bigg(\|f\|_{B^{\frac 1 4}_{2p,2}(B_\Lambda (0))} \|g\|_{B^{\frac 1 4}_{2p,2}(B_\Lambda (0))} \\ + \sum_{j =1}^\infty \frac {\|f\|^2_{L^{2p}(B_{\Lambda+j+1} (0) \setminus B_{\Lambda +j}(0))} +  \|g\|^2_{L^{2p}(B_{\Lambda+j+1} (0) \setminus B_{\Lambda +j}(0))}} {(\Lambda + j)^2}\Bigg). 
  \end{multline*}
 \end{lemma}

\begin{proof} 
  Remember that
  \begin{equation*}
   \frac{ \tilde Q f(x)}4:=  \;p.v. \! \int_{-\frac l2}^{\frac l2} \int_0^1 (1-s)\frac{ \kappa(x+sw) - \kappa (x)} {|w|^2}  dw = \tilde Q \kappa  (x)
  \end{equation*}

  Since
  \begin{align*}
   &\frac{\tilde Q^s[fg] - g \tilde Q^s[f]- f\tilde Q^s[g]} 4 \\ &= p.v. \int_{-\frac l 2}^{\frac l 2} \int_0^1 (1-s) \frac {f(x+sw) g(x+sw) -f(x) g(x) - (f(x+sw) -f(x) ) g(x) - (g(x+sw) -g(x)) f(x)}{|w|^2} dsw  \\&= p.v. \int_{-\frac l 2 }^{\frac l 2 } \int_0^1 (1-s) \frac {f(x+sw) g(x+sw) - f(x+sw) g(x) - f(x) g(x+sw) + f(x) g(x)} {|w|^2} dw \\ &=  \int_{-\frac l 2 }^{\frac l 2 } \int_0^1 (1-s) \frac {(f(x+sw) -f(x)) (g(x+sw) - g(x))} {|w|^2} dw  
  \end{align*}
  we get 
  \begin{align*}
   &\|\tilde Q^s[fg] - g \tilde Q^s[f]- f\tilde Q^s[g]\|_{L^p(B_1(0)}\\  &\leq C  \int_{-\frac l 2 }^{\frac l 2 } \int_0^1 (1-s)  \left(\int_{B_1(0)} \left(\frac {(f(x+sw) -f(x)) (g(x+sw) - g(x))} {|w|^2}  \right)^p dx \right)^\frac 1p dw \\
     &\leq C  \int_0^1  s (1-s)  \left(\int_{B_1(0)} \int_{-\frac l 2 }^{\frac l 2 }  \left(\frac {(f(x+w) -f(x)) (g(x+w) - g(x))} {|w|^2}  \right)^p dx \right)^\frac 1p dw    
   \\ & \leq C   \left(\int_{B_1(0)} \int_{-\frac \Lambda 2 }^{\frac \Lambda 2 }  \left(\frac {(f(x+w) -f(x)) (g(x+w) - g(x))} {|w|^2}  \right)^p dx \right)^\frac 1p dw \\ & \quad + C \left(\int_{B_1(0)} \int_{|w| \geq \frac \Lambda 2}  \left(\frac {(f(x+w) -f(x)) (g(x+w) - g(x))} {|w|^2}  \right)^p dx \right)^\frac 1p dw 
   \\ & \leq \|f\|_{B^\frac 1 2 _{4,2}(B_\Lambda(0)} \|g\|_{B^{\frac 1 2}_{4,2}(B_\Lambda(0)} + C \left(\int_{B_1(0)} \int_{|w| \geq \frac \Lambda 2}  \left(\frac {(f(x+w) -f(x)) (g(x+w) - g(x))} {|w|^2}  \right)^p dx \right)^\frac 1p dw
  \end{align*}
  Factoring out the product in the term and using the Cauchy inequality, we can estimate it by
  \begin{align*}
   &C  \int_{|w|\geq \Lambda} \left( \int_{B_1(0)}\frac{|f(x)|^{2p} + |g(x)|^{2p} } {w^2} dx \right)^{\frac 1p}dw  + \int_{|w|\geq \Lambda} \left( \int_{B_1(0)} \frac{|f(x+w)|^{2p}  + |g (x+w)|^{2p} } {w^2} dx \right) ^\frac 1p dw \\
   & \leq C \|f\|^2_{L^p(B_\Lambda(0)} \|g\|^2_{L^p(B_\Lambda(0)} + \sum_{j \in \mathbb N} \frac {\|f\|^2_{L^{2p}(B_{\Lambda + j} \setminus B_{\Lambda +j-1})} +\|g\|_{L^{2p} ( B_{\Lambda + j} \setminus B_{\Lambda +j-1})} } {(\Lambda +j)^2} 
  \end{align*}

\end{proof}

\newcommand{\etalchar}[1]{$^{#1}$}

%

\begin{thebibliography}{ACF{\etalchar{+}}03}

\bibitem[ACF{\etalchar{+}}03]{Abrams2003}
Aaron Abrams, Jason Cantarella, Joseph~H.G. Fu, Mohammad Ghomi, and Ralph
  Howard.
\newblock Circles minimize most knot energies.
\newblock {\em Topology}, 42(2):381 -- 394, 2003.

\bibitem[Bla11]{Blatt2011c}
Simon Blatt.
\newblock The gradient flow of the {M}\"obius energy near local minimizers.
\newblock {\em Calculus of Variations and Partial Differential Equations},
  pages 1--37, 2011.
\newblock 10.1007/s00526-011-0416-9.

\bibitem[Bla12]{Blatt2012}
Simon Blatt.
\newblock Boundedness and regularizing effects of {O}'{H}ara's knot energies.
\newblock {\em J. Knot Theory Ramifications}, 21(1):1--9, 2012.

\bibitem[Bla16]{Blatt2016}
Simon Blatt.
\newblock The gradient flow of {O}'{H}ara's knot energies.
\newblock January 2016, math.AP/1601.02840.

\bibitem[DKS02]{Dziuk2002}
Gerhard Dziuk, Ernst Kuwert, and Reiner Sch{\"a}tzle.
\newblock Evolution of elastic curves in {$\Bbb R\sp n$}: existence and
  computation.
\newblock {\em SIAM J. Math. Anal.}, 33(5):1228--1245 (electronic), 2002.

\bibitem[FHW94]{Freedman1994}
Michael~H. Freedman, Zheng-Xu He, and Zhenghan Wang.
\newblock M\"obius energy of knots and unknots.
\newblock {\em Ann. of Math. (2)}, 139(1):1--50, 1994.

\bibitem[Gas06]{Gastel2006}
Andreas Gastel.
\newblock The extrinsic polyharmonic map heat flow in the critical dimension.
\newblock {\em Adv. Geom.}, 6(4):501--521, 2006.

\bibitem[He00]{He2000}
Zheng-Xu He.
\newblock The {E}uler-{L}agrange equation and heat flow for the {M}\"obius
  energy.
\newblock {\em Comm. Pure Appl. Math.}, 53(4):399--431, 2000.

\bibitem[KS02]{Kuwert2002}
Ernst Kuwert and Reiner Sch{\"a}tzle.
\newblock Gradient flow for the {W}illmore functional.
\newblock {\em Comm. Anal. Geom.}, 10(2):307--339, 2002.

\bibitem[Lam04]{Lamm2004}
Tobias Lamm.
\newblock Heat flow for extrinsic biharmonic maps with small initial energy.
\newblock {\em Ann. Global Anal. Geom.}, 26(4):369--384, 2004.

\bibitem[LS10]{Lin2010}
Chun-Chi Lin and Hartmut~R. Schwetlick.
\newblock On a flow to untangle elastic knots.
\newblock {\em Calc. Var. Partial Differential Equations}, 39(3-4):621--647,
  2010.

\bibitem[Lun95]{Lunardi1995}
Alessandra Lunardi.
\newblock {\em Analytic semigroups and optimal regularity in parabolic
  problems}.
\newblock Progress in Nonlinear Differential Equations and their Applications,
  16. Birkh\"auser Verlag, Basel, 1995.

\bibitem[O'H91]{OHara1991}
Jun O'Hara.
\newblock Energy of a knot.
\newblock {\em Topology}, 30(2):241--247, 1991.

\bibitem[Tri83]{Triebel1983}
Hans Triebel.
\newblock {\em Theory of function spaces}, volume~78 of {\em Monographs in
  Mathematics}.
\newblock Birkh\"auser Verlag, Basel, 1983.

\bibitem[Tri92]{Triebel1992}
Hans Triebel.
\newblock {\em Theory of function spaces. {II}}, volume~84 of {\em Monographs
  in Mathematics}.
\newblock Birkh\"auser Verlag, Basel, 1992.

\end{thebibliography}
%
 \end{appendix}

\end{document}